\documentclass[reqno,11pt]{amsart}
\usepackage{times}
\usepackage{amsfonts,amsmath,amssymb}
\usepackage{mathrsfs,enumerate}
\usepackage[latin1]{inputenc}
\usepackage[T1]{fontenc}
\usepackage{color}
\numberwithin{equation}{section}
\newtheorem{theorem}{Theorem}[section]
\newtheorem{lemma}[theorem]{Lemma}
\newtheorem{remark}[theorem]{Remark}

\allowdisplaybreaks
\newcommand{\R}{\mathbb{R}}

\newcommand{\norm}[1]{\left\Vert#1\right\Vert}
\newcommand{\abs}[1]{\left\vert#1\right\vert}
\newcommand{\set}[1]{\left\{#1\right\}}
\newcommand{\para}[1]{\left(#1\right)}

\newcommand{\p}{\partial}
\newcommand{\LL}{\mathcal{L}}
\newcommand{\w}{\mathbf{w}}
\newcommand{\A}{A}

\newcommand{\g}{G}
\newcommand{\z}{\mathbf{z}}
\newcommand{\V}{\mathbf{v}}
\newcommand{\U}{\mathbf{u}}
\newcommand{\f}{\mathbf{f}}

\newtheorem{corollary}[theorem]{Corollary}

\begin{document}
\title{Stable reconstruction of the volatility in a regime-switching local volatility model}
\author{M. Bellassoued, R. Brummelhuis, M. Cristofol, E. Soccorsi}
\address{Mourad Bellassoued, University of Tunis El Manar, National Engineering School of Tunis, ENIT-LAMSIN, B.P. 37, 1002 Tunis, Tunisia. e-mail \texttt{mourad.bellassoued@enit.utm.tn } }
\address{Raymond Brummelhuis,   
Laboratoire de Math\'ematiques de Reims, Universit\'e de Reims, EA 4535, France. e-mail: \texttt{\ raymondus.brummelhuis@univ-reims.fr} }   
\address{Michel Cristofol, Institut de Math\'{e}matiques de Marseille, CNRS,
UMR 7373, \'Ecole Centrale, Aix-Marseille Universit\'e, 13453 Marseille,
France. e-mail: \texttt{\ michel.cristofol@univ-amu.fr} }   
\address{Eric Soccorsi, Aix-Marseille Univ., Universit\'e de Toulon, CNRS, CPT,  Marseille,
France.  e-mail: \texttt{\ eric.soccorsi@univ-amu.fr} }     
\date{}   
 
\begin{abstract}
Prices of European call options in a regime-switching local volatility model can be computed by solving a parabolic system which generalises the classical Black and Scholes equation, giving these prices as functionals of the local volatilities. We prove Lipschitz stability for the inverse problem of determining the local volatilities from quoted call option prices for a range of strikes, if the calls are indexed by the different states of the continuous Markov chain which governs the regime switches.   
\end{abstract}

\maketitle   

\section{Introduction and main result}
\label{sec-intro}

As is known since the fundamental work of Black, Scholes and Merton, prices of European options and other financial derivatives can be computed by solving a final value problem for a parabolic linear partial differential equation (PDE) known as the Black and Scholes equation. The coefficients of this PDE will depend on the stochastic model, in the form of a stochastic differential equation or SDE driven by one or more Brownian motions, for the underlying asset of the derivative. Conversely, for derivative contracts which are liquidly traded, one would like to use the market-quoted derivative prices to infer information about the (parameters of the) underlying stochastic model. Once the latter are known, the model can be used to price less liquidly traded derivatives, including derivatives which are not market-quoted but privately exchanged between two parties (the so-called `Over-the-Counter' contracts). Inferring the model from observed prices is known as the {\it calibration problem}. By what we have just said, it is clear that the calibration problem for derivatives can be interpreted as an {\it inverse PDE problem}, that of determining one or more coefficients of a parabolic PDE from the observation of certain of its solutions. Consequently, the various  techniques that have been developed to study inverse PDE problems can be made to bear on calibration. This point of view was first elaborated by Bouchouev and Isakov \cite{Bou-Isa1, Bou-Isa2}, who examined the problem of reconstructing the, by assumption time-independent, local volatility function of a so-called {\it Local Volatility model} from European call option prices of a fixed maturity but arbitrary strike, these prices being observed at some given point in time (and hence for some given value of the price of the underlying asset). We recall that a Local Volatility (LV) model for the price of a risky asset is Brownian motion-driven SDE in which the diffusion coefficient is itself a function of the asset-price and possibly also of time. Bouchouev and Isakov established unicity and, under additional assumptions, stability in H\"older norms.  An essential ingredient in their work was Dupire's PDE \cite{Dup} for European calls as functions of their strikes and maturities.   

In this paper we will prove a Sobolev-norm version of Bouchouev's and Isakov's stability estimate. Moreover, we will do so for a larger class of models called {\it Regime Switching Local Volatility or RSLV models}, also known as {\it Markov Modulated Local Volatility models}. These generalise the Local Volatility models by allowing the local volatility function to jump according to an independent continuous-time Markov chain, thereby incorporating features of a Stochastic Volatility or SV model.
\medskip   
   
The stochastic ingredients of a RSLV model are a Brownian motion $(W_t )_{t \geq 0 } $ and an independent finite state continuous time Markov process $(X_t )_{t \geq 0 } $ with state space\footnote{Some authors, following \cite{EAM}, realise the states $e_i $ as the canonical basis vectors of $\mathbb{R }^n $ (which amounts to identifying $e_i $ with the function $\epsilon _i $ introduced below). Functions on $E $ can then be identified with vectors in $\mathbb{R }^n $, the value of the function on $e_i $ being the  inner product with of the corresponding vector with $e_i $, and the process $(X_t )_{t \geq 0 } $ itself can be given the form of a semi-martingale (which would not make sense for an arbitrary set $E $). } $E = \{ e_1 , \ldots , e_n \} . $ The Markov chain is specified through its generator $B = (b_{ij } )_{1 \leq i, j \leq n } $ by
\begin{equation} \label{eq:MC1}
\mathbb{P } \left( X_{t + dt } = e_i | X_t = e_j \right) = \delta _{ij } + b_{ij } dt .
\end{equation}
Here $B $ is a matrix with non-negative off-diagonal elements whose column-sums are equal to zero: $\sum _i b_{ij } = 0 . $ We will call the $b_{ij } $ the transition probability rates, or more briefly, transition rates, of the continuous-time Markov chain $X_t . $   
Please note that some authors call $b_{ji } $ what we have called $b_{ij } $, which amounts to replacing $B $ by its transpose.

To specify the RSLV model, we are further given functions $r , q : E \to \mathbb{R } $ and $\sigma : \mathbb{R }_{\geq 0 } \times E \times \mathbb{R }_{\geq 0 } \to \mathbb{R } _{> 0 } $, to be interpreted as, respectively, the, Markov-state-dependent, interest and dividend rate\footnote{Both can in principle be negative: a negative $q $ would be a cost-of-carry, applicable if we are dealing commodities, and negative interest rates have occured in the recent past. } and the, equally state-dependent, local volatility function. If we let $S_t $ be the price at time $t $ of the underlying asset of the derivative, then our model for $S_t $ is:
\begin{equation} \label{eq:RSLV1}
dS_t = (r(X_t ) - q(X_t ) ) S_{t- } dt + \sigma (S_{t- } , X_t , t ) S_{t- } dW_t ,
\end{equation}
where $S_{t- } := \lim _{\varepsilon \to 0+ } S_{t - \varepsilon } $, complemented by the following pay-out rule for the dividend: the holder of the asset $S_t $ will receive $q(X_t ) S_t dt $ in dividends over the infinitesimal interval $[t , t + dt ] . $ It is important to note that we are working directly under given a risk-neutral probability measure, and not under the objective (or statistical) probability: indeed, under (\ref{eq:RSLV1}) the total expected return of the asset given $S_t $ due to price change plus dividend over an infinitesimal time interval $[t, t + dt ] $ is $r(X_t ) dt $, the risk-free return in that interval. We will often simply write $S_t $ instead of $S_{t- } . $   
   
The model (\ref{eq:RSLV1}) is interesting even when the volatility function $\sigma $ does not depend on $S $ and $t $, in which case we are dealing with a regime switching geometric Brownian motion model. Such a model can be regarded as a  simple type of stochastic volatility model, in which the stochastic driver of the volatility is a Markov chain instead of the Brownian motion which is traditionally used for stochastic volatility models, such as the models of Hull and White, of Stein and Stein and of Heston or the SABR model. Option pricing in regime-switching diffusion models is attracting an increasing interest in the literature: we mention \cite{Bollen} as an early reference, and as more recent ones \cite{Bu_Ell}, \cite{Chin_Du}, \cite{ECS}, \cite{ECS_2013}, \cite{ECS_2014}, 
\cite{FuHoHuWa}, \cite{Xi_Ro_Ma}, \cite{ZuBaLu}.      
\medskip

One possible, and popular, economic interpretation of the model is that the state space $E $ of the Markov chain represents different global states of the economy, for example, different phases of a business cycle. In each of these phases, which correspond to the different states $X_t = e_i $ of the Markov chain, the price $S_t $  of the underlying asset evolves according to a Local Volatility model with volatility function $\sigma _i (S, t ) := \sigma (S , e_i , t ) $ and interest rate and dividend rates $r_i := r(e_i ) $, $q_i := q(e_i ) . $ For example, the risk-free base rate, as set by a Central Bank, tends to be low during a recession, and high during an economic boom, and one would expect the opposite behaviour for the dividend rate.  In case $r : E \to \mathbb{R } $ is injective, investors can effectively observe the different states of the Markov chain. If not, we are dealing with a Hidden Markov model. In this paper we will assume that the Markov chain is observable. We will moreover assume that there exist European call options with a state dependent pay-off which again allows one to distinguish the different states. This last assumption will enable us to deduce a Dupire-type system of PDEs for European call options with coefficients given by the $\sigma _i $, $r_i $, $q_i $ and $b_{ij } . $   
\medskip

Turning therefore to derivatives, consider a European derivative paying its holder an amount of $F(S_T , X_T ) $ at maturity $T $, where $F : \mathbb{R }_{> 0 } \times E \to \mathbb{R} $ is a given function called the derivative's pay-off. (If $F (S_T , X_T ) $ takes on negative values these are interpreted as the amount the holder of the derivative will have to pay to the seller of the derivative, at maturity). By the fundamental theorem of asset pricing, and since the process $(S_t , X_t ) $ is Markovian, the value of the derivative at an earlier time $t < T $ is given by
\begin{equation}
V (S_t , X_t , t ) := \mathbb{E } \left( \, e^{ - \int _t ^T r(X_u ) du } F (S_T , X_T ) \, | S_t , X_t \right) ,
\end{equation}   
where the expectation is with respect to the risk-neutral measure and where $V $ is a real-valued function on $\mathbb{R }_{> 0 } \times E \times \mathbb{R } . $ If we let $V_j (S , t ) := V(S, e_j , t ) $, then an application of Ito's lemma shows that the $V_j $'s satisfy the following Black and Scholes-type system of PDEs for $t < T $:
\begin{equation} \label{eq:RSLV_BSE1}
\partial _t V_j + \frac{1 }{2 } \sigma _j (S, t) ^2 S^2 \partial _S ^2 V_j + (r_j - q_j ) S \partial _S V_j + \sum _i b_{ij } V_i = r_j V_j , \ \ j = 1 , \ldots , n .
\end{equation}
It is convenient to introduce the {\it row}-vector ($=1 \times n $ matrix) $\underline{V }(S, t ) := \left(V_1 (S , t ) , \ldots , V_n (S , t ) \right) $ as well as the diagonal $n \times n $-matrices $\Sigma (S, t ) := {\rm diag } \left( \sigma _1 (S, t ) , \ldots \sigma _n (S, t  ) \right) $, $R  := {\rm diag }(r_1 , \ldots , r_n ) $ and $Q := {\rm diag }(q_1 , \ldots , q_n ) $ and write the system as      
\begin{equation} \label{eq:RSLV_BSE2}
\underline{V }_t + \frac{1 }{2 } S^2 \underline{V }_{SS } \, \Sigma (S, t )^2 + S \underline{V}_S \, (R - Q ) + \underline{V } \, (B - R ) = 0 ,   
\end{equation}
where $\underline{V }_t := \partial _t \underline{V } $ and similar for $\underline{V }_S $ and $\underline{V }_{SS } $ and where $B $ is the matrix of transition probabilities of the Markov chain which was already defined above. (Note that these matrices multiply $\underline{V } $ and its derivatives on the right.)   
This system has to be supplemented by the final boundary condition
\begin{equation} \label{eq:RSLV_BSE3}
\underline{V }(S, T ) = \underline{F }(S ) ,   
\end{equation}
where $\underline{F }(S) $ is the row-vector $(F_1 (S ) , \ldots , F_n (S) ) $ with $F_j (S) := F (S , e_j ) . $ If $n = 1 $, the Markov chain is trivial, and (\ref{eq:RSLV_BSE1}) reduces to the classical Black and Scholes equation for a Local Volatility model.   
\medskip   

We briefly discuss unique solvability of (\ref{eq:RSLV_BSE2}) - (\ref{eq:RSLV_BSE3}). We will assume that there exists two positive constants $\sigma _{min } $ and $\sigma _{max } $ such that   
\begin{equation}
\label{1.2}
0 < \sigma_{\min} \leq \sigma _j (S, t ) \leq \sigma_{\max} < \infty ,   
\end{equation}   
for all $j = 1 , \cdots , n $ and $S, t  > 0 . $ In particular, the system (\ref{eq:RSLV_BSE2}) will be parabolic. Since $S $ is constrained to be positive we would, at first sight, need an additional boundary condition on $S = 0 . $ Note, however, that the principal symbol of (\ref{eq:RSLV_BSE2}) is not uniformly parabolic  on $\{ S > 0 \} $, but degenerates at the boundary, and a  logarithmic transformation $S = e^x $ and time-reversal $\tau := T - t $ transforms our problem into a Cauchy problem for a uniformly parabolic system on $\mathbb{R }^n \times \mathbb{R }_{> 0 } $: if we let $\underline{v } (x, \tau ) := \underline{V } (e^x , T - \tau ) $ and $A(x, \tau ) := \frac{1 }{2 } \Sigma ^2 (e^x , T - \tau )  = {\rm diag } (a_1 (x, t ) ,  \ldots , a_n (x, t ) ) $ with $a_j (x, \tau ) = \sigma _j (e^x , T - \tau ) $, then\begin{equation} \label{eq:RSLV_BSE4}   
\left \{ \begin{array}{ll}   
\underline{v }_{\tau }  = \underline{v }_{xx } \, A (x, \tau ) + \underline{v }_x \, \left( R - Q - A(x, \tau ) \right) + \underline{v } \, (B - R ) = 0 \\   
\underline{v } (x, 0 ) = \underline{f }(x) ,   
\end{array}   
\right.   
\end{equation}   
where $\underline{f }(x) := \underline{F }(e^x ) . $ By the classical theory of parabolic systems (see for example \cite{Eidelman} or \cite{F}) if, in addition to (\ref{1.2}), the $a_j (x, \tau ) $ are continuous,  bounded and $C^{\alpha } $ in $x $ for some $0 < \alpha < 1 $, uniformly in $(x, \tau ) $, and if there exist constants $c, C > 0 $ such that   
\begin{equation} \label{eq:RSLV_BSE5}
|f_j (x) | \leq C e^{c|x|^2 } , \ \ j = 1 , \ldots , n ,
\end{equation}
then (\ref{eq:RSLV_BSE4}) will have a unique\footnote{Since our system is diagonal in the first and second order terms, we can use the maximum principle (more precisely, the Phragmen - Lindel\"of principle) to prove uniqueness. For more general systems, uniqueness follows from solvability of the adjoint system, but for that the coefficients would need to be $C^{2, \alpha } . $ }  solution subject to the  components of  $\underline{v } $ also satisfying (\ref{eq:RSLV_BSE5}), uniformly in $\tau . $ We note that assuming (\ref{1.2}), the H\"older condition on $A(x, \tau ) $ is equivalent to   
$$
\sup _{S_1 , S_2 , \tau > 0 } \frac{|\sigma _j (S_1 , \tau ) - \sigma _j (S_2 , \tau ) | }{|\log (S_2 / S_1 ) |^{\alpha } } < \infty , \ \ j = 1 , \ldots , n .   
$$    
\medskip   
   
We next specialise to a class of generalised European call options whose pay-offs are, by definition, of the form $F(S , X ) := \max (S - K , 0 ) \pi (X ) $, where $\pi : E \to \mathbb{R } $ is a given function. The pay-off at $T $ therefore in principle depends on the Markov chain state the economy will be in at $T . $ Including the pay-off parameters $K $, $\pi $ and $T $ in the notation, we will denote the resulting call option prices by row-vectors   
$$   
\underline{C } (K, T, \pi ; S , t ) = \left( C_j (K, T, \pi ; S, t ) \right)_{1 \leq j \leq n } ,   
$$   
whose $j $-th component is the price at $t $ when the state of the Markov chain at $t $ is $e_j $:   
$$   
C_j (K, T, \pi ; S, t ) := C(K , T, \pi ; S , e_j , t ) .   
$$   
Here, the option parameters $K $, $T $, $\pi $ should be thought of as forward variables. As a function of $S $ and $t $, the vector   
$\underline{C } (K, T, \pi ; S , t ) $ solves (\ref{eq:RSLV_BSE2})-(\ref{eq:RSLV_BSE3}) with $F_j (S ) := \pi _j \max (S - K , 0 ) $, where $\pi _j := \pi (e_j ) . $ Since the pay-off of the option at $T $ is between $(\min _j \pi _j ) S $ and $(\max _j \pi _j ) S $, the same has to be true for each of the $C_j $, since otherwise the market would present arbitrage opportunities. In particular, the $ | C_j (e^x, t ) | \leq {\rm Const.} \, e^x $ and $\underline{C } (K, T , \pi ; S , t ) $ is the unique solution of (\ref{eq:RSLV_BSE2}) with final pay-off $\max (S - K , 0 ) \underline{\pi } $, where $\underline{\pi } $ is the row vector with components $\pi _j . $   
\medskip   
   
We now suppose that sufficiently many of such generalised call options are traded, in the sense that the associated $\underline{\pi } $'s will span $\mathbb{R }^n . $ By taking suitable linear combinations, we can then take $\pi = \epsilon _i $, $i = 1 , \ldots , n $, where  $\epsilon _i : E \to \mathbb{R } $ is defined by
$$   \epsilon _i (e_j ) = \delta _{ij } ;
$$
equivalently, $\epsilon _i $ is the indicator function of the singleton subset $\{ e_i \} \subset E . $ The associated call option will pay out $\max (S_T - K , 0 ) $ if the economy is in state $i $ at $T $,  and nothing if the economy finds itself in one of the other states. Such securities effectively provide a set of Arrow-Debreu securities for the economy, allowing investors to not only differentiate between different price levels of the underlying asset $S $ at maturity $T $, but also between the different Markov-chain states. A classical European call, which corresponds to taking for $\pi $ the function which is identically 1, does not allow this, since it will pay off the identical amount $\max (S_T - K , 0 ) $ regardless of the state $X_T . $

It is convenient to collect all possible call option values into a matrix-valued function $\mathcal{C }(K, T ; S , t ) = \left( C_{ij } (K, T ; S , t ) \right) _{1 \leq i, j \leq n } $, where
\begin{equation}
C_{ij } (K , T ; S , t ) = C (K, \epsilon _i , T ; S , e_j , t ) ,   
\end{equation}   
is the value at $t $ of the generalised call with pay-off $\max (S - K , 0 ) \epsilon _j $ at $T $, when $S_t = S $ and $X_t = e_j . $ This function   
$\mathcal{C } $ then satisfies
\begin{equation} \label{eq:RSLV_BSE6a}   
\left \{ \begin{array}{cc}   
\mathcal{C }_t + \frac{1 }{2 } S^2 \mathcal{C }_{SS }   
\Sigma (S , t )^2 + S \mathcal{C }_S   
(R - Q ) + \mathcal{C }      
(B - R ) = 0 \\   
\mathcal{C } (K, T ; S , T ) = \max (S - K , 0 ) I_n ,   
\end{array}   
\right.   
\end{equation}
with $I_n $ the $n \times n $ identity matrix. We will see in section 2 below that, as function of $K $ and $T $, $\mathcal{C } $ satisfies a closely related system which we will call the Dupire system and which reduces to the well-known Dupire equation when $n = 1 . $   

The instantaneous volatility function $\sigma (S, X, t ) $ (momentarily reverting to the notation (\ref{eq:RSLV1})) cannot be directly observed, though we can in principle infer the realised volatilites $\int _t ^{t + h } \sigma (S_u , X_u , u )^2 du $, as a process in $t $, from historical data, for different time-windows $h . $ Inferring $\sigma (S, X, t )^2 $ from this is a non-trivial statistical problem and would, at best, provide information about this function for values of $S > 0 $ and $X \in E $ which have been realised sufficiently often. This might not include values which are relevant for the, by their nature future-looking, call options. As regards the latter, the Regime-Switching Black-Scholes model (\ref{eq:RSLV_BSE6a}) defines European call prices as a (non-linear) function of $\Sigma $, and if European call prices are liquidly quoted in the market, one may hope that such call prices determine the volatility. This is indeed the message of the classical Dupire equation, and of its generalisation to a RSLV model in section 2 below, in the ideal case that European call option prices quoted for the full range of possible maturities, strikes, and final Markov-chain states. In mathematical terms this means that the volatility can be identified from call option prices in a RSLV model. In practice, call prices will of course only be available for a finite number of strikes and maturities, and a naive use of Dupire's equation, in combination with some interpolation procedure to construct prices for the missing strikes and maturities, often leads to numerically unstable reconstructed volatility functions. This is not surprising since we are dealing with a typical ill-posed inverse PDE problem. Is is therefore of interest to dispose of some type of stability estimate on which to base a sound numerical procedure. Establishing such an estimate is the main objective of this paper.   
\medskip

We will suppose from now on,  following \cite{Bou-Isa1}, \cite{Bou-Isa2}, that the volatility is a function of $S $ only, $\Sigma (S ) = {\rm diag} \left( (\sigma _1 (S ) , \ldots , \sigma _n (S) \right) $: available maturities for which option prices are available are typically sparse, and one can concentrate on option prices with a single maturity, to construct volatility functions which piece-wise constant as function of time: cf. \cite{Bou-Isa1} and \cite{Bou-Isa2}.   

 To state the necessary smoothness conditions on $\Sigma $, we introduce the space $C^k _b (\mathbb{R } ) $ of $k $-times continuously differentiable functions with uniformly bounded derivatives and its subspace $C_b ^{k , \alpha } (\mathbb{R } ) $ of functions whose derivative of order $k $ in addition satisfies a uniform H\"older condition of order $\alpha \in (0, 1 ) . $ We also need the Sobolev spaces $H^1 (I) , H^2 (I ) $ where $I \subset \mathbb{R } $ is an open interval. Vector- and matrix-valued functions will be said to belong to these spaces if all of their components do. As concerns the norms of the latter, we will use the following notational convention  : for any Hilbert or Banach space $X $ equipped the norm $\|\cdot\|_X$, and any $\z=(z_1,\dots,z_n ) \in X^n $ we put
\begin{equation} \label{eq:convention_vector-norm}   
\|\z\|_X^2 := \sum_{j=1 }^n \|z_j\|_X^2 ,   
\end{equation}   
instead of the more precise but more cumbersome $\| \z \| _{X^n . } $    
\medskip

We will also need the continuous-time Markov chain $(X_t )_{t \geq 0 } $ to be {\it irreducible}, which means that for all $i \neq j $ there exists a $t > 0 $ such the transition probability $\mathbb{P }(X_t = e_i | X_0 = e_j ) > 0 $: see for example \cite{Davies}. The matrix of transition probabilities is known to be $e^{t B } $ (as follows by integrating the forward Kolmogorov equation), and irreducibility of the Markov chain can be shown to be equivalent to all of the matrix elements of $e^{tB } $ being strictly positive for all $t > 0 $, which is the form in which we will use it. A Markov chain will certainly be irreducible if $b_{ij } > 0 $ for all $i \neq j $, but this condition is not necessary. 
\medskip   
   
Let $C_{\Sigma } (K , T ; S , t ) $ be the matrix of generalised European call options defined by (\ref{eq:RSLV_BSE6a}), where we make the dependence on the volatility matrix $\Sigma $ explicit. We seek to stably reconstruct $\Sigma (S ) $ on some given interval $I \subset (0, \infty ) $ from observed call prices of a fixed maturity $T $ and with strikes $K $ varying in a slightly larger interval $J \Supset I . $ If we observe these call prices at some fixed time $t^* $ at which the Markov chain will be state $X_{t^* } = e_{j^* } $, then we will know the $j^* $-th column of $C_{\Sigma } (K , T ; S^* , t^* ) $ where $S^* := S_{t^* } $, that is,   
\begin{equation}\label{1.6}
C_{\Sigma } ^* (K, \epsilon _i ) := C_{\Sigma } (K , T , \epsilon _i ; S^* , e_{j^* } , t^* ) , \ \ i = 1, \ldots , n .   
\end{equation}   
The first main result of this paper is the following local stability estimate which in the case of $n = 1 $ should be compared with Theorem 1 of \cite{Bou-Isa2}, which gives a similar estimate for the $C^{\alpha } $-norm of $\Sigma _1 - \Sigma _2 $ in terms of the $C^{2, \alpha } $-norm of $C^*_{\Sigma_1 } - C^*_{\Sigma_2 } $. 
Such a stability inequality is helpful for designing the iteration method used in the numerical reconstruction of the unknown volatility matrix, see e.g. \cite{VdHQiu12}. We recall our simplified notation  (\ref{eq:convention_vector-norm}) for norms of elements of vector-valued function spaces. 
   
\begin{theorem}\label{T1.1}
Suppose that the Markov chain $(X_t )_{t \geq 0 } $ is irreducible, and that the coefficients of the (diagonal) matrix   
$\Sigma _1 (e^y ) $ are in $C_b ^1 (\mathbb{R } ) $, while those of   
$\Sigma _2 (e^y ) $ are in $C_b ^{2, \alpha } (\mathbb{R } ) $ for some $\alpha > 0 $,   
with both satisfying (\ref{1.2}). Assume that
\begin{equation} \label{eq:thm1.1}
\Sigma_1 = \Sigma _2 \ \mbox{on $\mathbb{R } \setminus I $, }
\end{equation}
on some bounded interval $I \Subset (0, \infty ) . $ Then for any bounded open interval $J \Subset (0, \infty ) $ such that $I \Subset J $ there exist a constant $C>0 $, depending only on $\sigma _{\rm min } $, $||\Sigma _1 ||_{C^1 (J) } $, $||\Sigma _2 ||_{C^{2, \alpha } (\mathbb{R } ) } $, $I$, $J$ and $\tau _* := T - t^* $, such that we have
\begin{equation} \label{eq:thm1.2}
\norm{\Sigma_1 - \Sigma_2}_{L^2(I) } ^2 \leq C \sum _{i = 1 } ^n \ \norm{C^*_{\Sigma_1 } (\cdot , \epsilon _i ) - C^*_{\Sigma_2 } (\cdot , \epsilon _i ) }_{H^2(J ) } ^2 .
\end{equation}
\end{theorem}

\noindent Note the different regularity assumptions on $\Sigma _1 $ and $\Sigma _2 . $ The reason is technical: at certain points in the proof we use Gaussian bounds for the fundamental solution of a parabolic PDE which is essentially the adjoint of (\ref{eq:RSLV_BSE4}) with $A = \frac{1 }{2 } \Sigma _2 ^2 $, and which are only known for parabolic PDEs with coefficients which are uniformly $C^{\alpha } . $ We can think of (\ref{eq:thm1.2}) as a first order Taylor estimate for the inverse function $C_{\Sigma } ^* \to \Sigma $ at a  point $C_{\Sigma _2 } ^* $ which is the image of a sufficiently regular $\Sigma _2 . $ It may be that both $\Sigma _i $ having coefficients in $C^1 _b (\mathbb{R } ) $ would suffice, but $C^1 $ is  the minimum regularity required, at least for our proof which is based on a Carleman estimate.   
\medskip

We also note that some condition on the Markov chain is necessary, as shown by the trivial example of $B = 0 $, for which the system decouples. In that case, $C^* _{\Sigma } (K, T ) $ will at best allow reconstruction of $\sigma _{j^* } (S ) $, but not of the other elements of $\Sigma (S ) $, as can be seen from the generalised Dupire equation (\ref{eq:RS_Dupire3}) below.
\medskip   
   
If we do not want to assume that $\Sigma _1 $ and $\Sigma _2 $ coincide outside of an interval $I $, we have the following version of (\ref{eq:thm1.2}).

\begin{theorem} \label{thm:1.2}   
Suppose that $\Sigma _1 $, $\Sigma _2 $ and $(X_t )_{t \geq 0 } $ satisfy the conditions of theorem \ref{T1.1} except for (\ref{eq:thm1.1}), which is replaced by
\begin{equation}
\Sigma _1 - \Sigma _2 \in L^2 \left( \mathbb{R }_{> 0 } , \frac{dS }{S } \right) ^n .   
\end{equation}
Let $I \Subset J \Subset (0, \infty ) $ be bounded intervals with $J $ open, such that $S^* \notin I . $ Then there exists a constant $C  > 0 $, with the same dependence as in theorem \ref{T1.1}, such that
\begin{eqnarray} \nonumber
\norm{\Sigma_1 - \Sigma_2}_{L^2(I) } ^2 &\leq & C \left( \sum _{i = 1 } ^n \ \norm{C^*_{\Sigma_1 } (\cdot , \epsilon _i ) - C^*_{\Sigma_2 } (\cdot , \epsilon _i ) }_{H^2(J ) } ^2 + || \Sigma _1 - \Sigma _2 ||^2 _{H^1 (J \setminus I ) } \right. \\
&& \left. + || \Sigma _1 - \Sigma _2 ||^2 _{L^2 (\mathbb{R }_{> 0 }   
\setminus I , dS / S ) } \right) . \label{eq:thm1.2a}
\end{eqnarray}
\end{theorem}

Theorem \ref{thm:1.2} implies for example that (\ref{eq:thm1.2}) will be true modulo an error of the order of $\varepsilon $, if both $|| \Sigma _1 - \Sigma _2 ||_{H^1 (J \setminus I ) } $ $ || \Sigma _1 - \Sigma _2 ||_{L^2 (\mathbb{R }_{> 0 } \setminus I , dS / S ) } $ are smaller than $\varepsilon . $ It also shows that if, for a given $\Sigma (e^y ) \in C_b ^{2 , \alpha } (\mathbb{R } ) $, we can find a sequence $\Sigma _{\nu } (e^y ) \in C_b ^1 (\mathbb{R } ) $ such that $\norm{C^*_{\Sigma } (\cdot , \epsilon _i ) - C^*_{\Sigma_{\nu } } (\cdot , \epsilon _i ) }_{H^2(J ) } $, $|| \Sigma - \Sigma _{\nu } ||_{H^1 (J \setminus I ) } $ and $|| \Sigma - \Sigma _{\nu } ||_{L^2 (\mathbb{R }_{> 0 } \setminus I , dS / S ) } $ all tend to 0, while $|| \Sigma _{\nu } ||_{C^1 (J ) } $ remains bounded, then $\Sigma _{\nu } \to \Sigma $ in $L^2 (I ) $ as $\nu \to \infty . $   
\medskip   
   
We end this introduction with some general remarks to situate our work within the existing literature. We already noted that this paper contributes to the growing literature on Markov-modulated diffusion processes and their applications in Finance.   
The problem of reconstructing a local volatility function from observed option prices has attracted a lot of attention in the mathematical finance literature, ever since the publication of Dupire's paper \cite{Dup}. Without claiming to be exhaustive, we cite \cite{Bou-Isa1} and \cite{Bou-Isa2} for stability and uniqueness results in H\"older norm, \cite{Crepey} and \cite{WaYaZe} for Tykhonov-style approaches using different types of regularisation, \cite{L01} and \cite{Den08} for an optimal control approach, \cite{Isa02} and \cite{Isa14} for an approach based on linearisation around a constant volatility, and \cite{Cristo_Roques} where the authors establish a uniqueness result using a formulation of the problem in terms of stochastic differential equations: they prove a one-to-one relationship between the diffusion coefficient of a diffusion process $X_t $ and the expectation of strictly concave or convex function of $X_t $ observed during a small time interval.   

 All of these papers only consider the scalar case (corresponding to $n = 1 $), without regime switching. Our paper is closest to, and inspired by, Bouchouev and Isakov \cite{Bou-Isa1} and \cite{Bou-Isa2}, but we establish estimates in Sobolev instead of H\"older space norms. Our proof of theorems \ref{T1.1} and \ref{thm:1.2} uses the  Bukhgeim-Klibanov method, as in Imanuvilov and Yamamoto's \cite{[IY2]}, though we will only need a Carleman estimate which is local in space. The system of PDEs we have to work with is only coupled in the 0-th order terms, and diagonal in the first and second order. This makes that much of the analysis for the scalar Dupire PDE carries through automatically, but we do need at one point  strict positivity of the matrix coefficients of the fundamental solution, which is true iff the underlying Markov chain of the model is irreducible. This result, which we prove in Appendix B, seems of some interest in its own right. Another technical problem occurs because of the non-smoothness of the call option pay-off at the strike $K $ which, after linearisation of the inverse problem, leads to non-homogeneous parabolic problem whose right hand side has a fundamental-solution type singularity at time 0: see section 4. This singularity prevents us from applying standard parabolic energy estimates where we would have liked to, and we have had to replace these by certain weighted energy estimates derived in section 3, as well as by Schur -type estimates for the norms of certain integral operators in Appendix C.\\
   
\section{The Dupire equation}   
   
The inverse problem of determining the local volatilities $\sigma _i (S) $ from observed option prices can be reformulated as a coefficient identification problem for a parabolic system of PDEs in $K $ and $T . $ This is a consequence of the following result, which generalizes the well-known Dupire equation \cite{Dup} to the regime-switching case.

\begin{theorem} \label{thm:RS_Dupire} Suppose that the function $A(x, \tau ) := \Sigma (e^x  , T - \tau ) $ is uniformly elliptic, continuous and bounded and $C^1 $ in $x $, with uniformly bounded derivative. Then, as function of $K $ and $T $, the matrix of call option prices $\mathcal{C } $ satisfies the forward parabolic system
\begin{equation} \label{eq:RS_Dupire1}
\mathcal{C }_T = \frac{1 }{2 } K^2 \Sigma (K, T )^2 \mathcal{C }_{KK } - K (R - Q ) \mathcal{C }_K + (B - Q ) \mathcal{C } , \ \ T > t ,
\end{equation}
with boundary condition
\begin{equation} \label{eq:RS_Dupire2}
\mathcal{C } (K, t ; S, t ) = \max (S - K , 0 ) \, I_n ,   
\end{equation}   
at $T = t . $   
\end{theorem}

We refer to Appendix A for a proof. We note that in \cite{Xi_Ro_Ma} the authors consider the inverse problem in the case when the volatility matrix is a function $\Sigma = \Sigma (t) $ of time only and which, in a certain sense, is complementary to the case considered in this paper. They established a Dupire equation for classical call options with state-independent pay-offs that is different from ours,  and which can be derived by observing that if the volatilities $\sigma _i $ do not depend on $S $, then the call option prices $C_j = C (K , T , \mathbf{1 } ; S, t , e_i ) $   
are homogeneous of degree 1 in $(S, K ) . $ The Euler relations for homogeneous functions then enables one to re-write the Black and Scholes system (\ref{eq:RSLV_BSE6a}) as a system of PDEs in $K $ and $T . $ Multiplying on the left by the row-vector $\mathbf{1 } $ we then find for this case a Dupire-type system satisfied by {\it classical call} options $C_j = C (K , T , \mathbf{1 } ; S, t , e_i ) . $ For general dependend $\Sigma (S , t ) $ this does not seem to be possible: in equation (\ref{eq:RS_Dupire1}), the matrix $\mathcal{C } $ and its derivatives multiply the coefficients on the left instead of on the right, so left-contraction with the row-vector $\mathbf{1 } $ is no longer possible.    
\medskip   
   
If we now fix $S^* > 0 $, $t^* > 0 $ and $j^* \in \{ 1 , \ldots , n \} $, then the   
column vector   
$$   
\U (K, T ) := \left( C(K , T , \epsilon _i  ; S^* , t^* , e_{j^* } \right)_{1 \leq i \leq n } ,   
$$   
(which is simply the $j^* $-th column of $\mathcal{C } $) satisfies the parabolic system
\begin{equation} \label{eq:RS_Dupire3}
\partial _T \U = \frac{1 }{2 } K^2 \Sigma (K, T )^2 \U _{KK } - K (R - Q ) \U _K + (B - Q ) \U , \ \ T > t^* ,
\end{equation}
with boundary condition $\U (K, t^* ) = \max (S^* - K ) e_{j^* } $, where from now we will identify $e_j $ with the $j $-th canonical basis vector of $\mathbb{R }^n . $ The logarithmic substitution
$$
y = \ln (K/S^* ) ,   
\ \tau=T -t^* , \ \w_\A(y,\tau) =\mathbf{u}(K,T) ,   
$$
transforms the initial value problem for $\U (K, T ) $ into
\begin{equation}\label{2.6}
\left\{\begin{array}{ll}
\para{\p_\tau-\mathcal{L}_\A}\w_\A=0, & \textrm{in}\ \R \times (0,\infty), \cr
\w_\A(y,0) = S^*\max(1-e^y, 0 ) e_{j^* }  , & y\in\R, \cr   
\end{array}
\right.
\end{equation}
where
\begin{equation}
\label{defLa}
\LL_\A\w = \A(y)\w_{yy}-\para{\A(y)+R-Q}\w_y + (B - Q ) \w,
\end{equation}
with $\, \A(y) := \frac{1}{2} \Sigma^2 \left( e^y \right) $, and where we made the dependence of the solution on the coefficient $A $ explicit. This initial value problem is uniquely solvable if we require that the solution $\w _{\A } $ be bounded (note that the solution is unique if it satisfies the bound (\ref{eq:RSLV_BSE5}), but since the initial value is bounded, $\w _{\A } $ will then automatically be bounded, by the maximum principle for weakly coupled parabolic systems \cite{PW}).   
\medskip   
   
Our inverse problem can therefore be rephrased as follows: letting
$$
\Omega _1:= \set{ \ln(K/S^*),\ K \in I } \ \Subset \Omega := \set{ \ln(K/S^*),\ K \in J } ,
$$
can we stably retrieve $\A(y) $ on $\Omega _1 $ from knowing $\w_\A(y,\tau _* ) $ on $\Omega $, where $\tau _*  := T - t^* $ and $\w_\A $ is the unique bounded solution of (\ref{2.6})?

\section{Parabolic estimates for the Dupire equation}
\label{sec:Carleman}

Our proof of theorem \ref{T1.1} uses a Carleman estimate for parabolic equations which is local in the space-variable $y . $ To state this estimate, let $\psi \in C^2 (\overline{\Omega } ) $ with $\psi '(y) \neq 0 $ on $\overline{\Omega } . $ We put  
$T^* := 2 \tau _*  $ and $\ell (\tau ) := \tau (T^* - \tau ) . $ Next, for $L > ||\psi ||_{\infty } $, define   
$$
\eta(y, \tau ) := \eta_\lambda(y, \tau ) :=\frac{e^{\lambda\psi(y)}-e^{\lambda L }}{\ell(\tau)},\ y \in \Omega,\ \tau\in (0, T^* ) ,
$$
where $\lambda $ will be chosen sufficiently large below. Note that $\eta (y, \tau ) < 0 $ on $Q := \Omega \times (0, T^* ) $, and that $\eta (y, \tau ) \to -\infty $ as $\tau \to 0 , T^* . $   

Let
$$
H^{2,1}(Q)=\set{z\in L^2(Q),\,z_y,\,z_{yy}, z_{\tau}\in L^2(Q ) },   
$$
be the usual Sobolev parabolic space, endowed with the norm
$$
\norm{z}_{H^{2,1}(Q)} ^2 =\norm{z}_{L^2(Q)} ^2 +\norm{z_y}_{L^2(Q)} ^2 +\norm{z_{yy}}_{L^2(Q)} ^2 +\norm{z_\tau}_{L^2(Q)} ^2 ,
$$
which we extend in the natural way to vector-valued functions via (\ref{eq:convention_vector-norm}). We next state the parabolic Carleman estimate which we will use and which, since the principal symbol of $\mathcal{L }_A $ is diagonal, it is an immediate consequence of a scalar Carleman estimate established in \cite{[FI],[IY2]}; see also \cite{LR_L}, theorem 7.3 for an approach using semi-classical analysis.   
   
\begin{theorem}\label{L3.1} Suppose that $A = ( a_i \delta _{ij } )_{i, j } $ is diagonal with $a_i \in C^1 (\Omega ) $, for all $i $,and let $K \subset \Omega $ be compact. Then there exists a $\lambda_0 > 0$ such that, for all $\lambda > \lambda _0 $, we may find constants $s_0> 0$ and $C>0$ such that, for all $\z=(z_1,\ldots,z_n ) \in H^{2, 1 } (Q ) ^n $ with ${\rm supp }(z_i) \subset K \times [0, T^* ] $ and all $s \geq s_0 $, we have   
\begin{equation}\label{eq:Carleman_est}   
\| \left( \frac{s }{\ell } \right)^{3/2 } e^{s\eta} \, \z \, \|_{L^2 (Q) } ^2 + \| \left( \frac{s }{\ell } \right)^{1/2 } e^{s\eta } \, \z_y \, \|_{L^2 (Q) } ^2 +
\| \left( \frac{s }{\ell } \right) ^{- 1 / 2 } e^{s\eta } \, \z_\tau  \, \|_{L^2 (Q) } ^2 \\   
\leq C \| e^{s\eta} ( \p_\tau-\LL_\A ) \, \z \, \|_{L^2 (Q) } ,   
\end{equation}   
where the constant $C$ depends on the $C^1 $-norm of $A $ and on $\lambda $ but is independent of $s . $   

\end{theorem}

 We will also need a technical result in the form of a  local parabolic regularity estimate with Carleman weights, which we state in slightly greater generality than needed. Consider a parabolic $n \times n $ system
\begin{equation} \label{eq:gen_para_system}   
\p_\tau\U - \p_y(C_2 \p_y\U) - C_1 \p_y\U - C_0 \U = \f ,   
\end{equation}   
on $Q = \Omega \times (0, T^* ) $ with real-matrix valued coefficients $C_i = C_i (y) $ and real symmetric positive-definite $C_2 = C_2 (y) . $ Solutions will be interpreted in the either the strong or the weak sense, depending on context, and we will limit ourselves to solutions with real-valued components.   
\medskip   

\begin{lemma} \label{L3.2} Suppose that the coefficients $C_0 $ and $C_1 $ of (\ref{eq:gen_para_system}) are locally bounded on $\Omega $ and that $C_2 $ is of class $C^1 $ there, while $\f \in L^2 _{\rm loc } (Q ) ^n . $  Let $\gamma \in [0,+\infty)$  and let $\Delta \Subset D  \Subset \Omega . $ Then there exists a constant $C=C(\gamma,\Delta,D)>0$ such that for all  $\U \in H^{2,1}(Q ) ^n $ solution of (\ref{eq:gen_para_system}) and all $s\geq s_0 > 0 $,
\begin{equation} \label{3.2a}
\| e^{ s \eta} \ell^{-\gamma } \U_y \|^2_{L^2(\Delta \times (0,T^* ))} \leq C \left(  s^2 \| e^{s \eta} \ell^{- \gamma -1 } \U \|^2_{L^2(D \times (0,T^* ) ) } + s^{-2 } \| e^{s \eta} \ell^{-\gamma + 1 } \f \|^2_{L^2(D \times (0,T^* ))} \right),
\end{equation}
If, moreover, $\f _y \in L^2 _{\rm loc } (Q ) ^n $ and $C_0 $ is $C^1 $ also, then   
\begin{eqnarray} \label{3.2b}
\| e^{s \eta} \ell^{-\gamma} \U_{\tau} \|^2_{L^2(\Delta \times (0,T^* )) }
& \leq & C \left(  s^4  \| e^{s \eta}\ell^{- \gamma - 2  } \U \|^2_{L^2(D \times (0,T^* ))} + \| e^{s \eta} \ell^{- \gamma } \f \|^2_{L^2(D \times (0,T^* ))} \right. \nonumber \\
& & \hspace*{1cm} \left. + s^{-2 } \| e^{s \eta} \ell^{-\gamma  + 1  } \f_y \|^2_{L^2(D \times (0,T^* ))} \right).   
\end{eqnarray}
\end{lemma}

\begin{proof}
Let $0 \leq \chi  \in C_c ^{\infty}(\R ) $ be supported in $D$ with $\chi(y)=1$ for all $y \in \Delta $, and put $w := \ell^{ - 2 \gamma } \chi e^{2 s \eta } . $ Then, with $( \cdot , \cdot ) $ denoting the inner product of $L^2 (Q ) ^n $,   
\begin{equation} \label{eq:L3.2_1}
- \left( (C_2 \U _y )_y , w \U \right) = (\f , w \U ) - (\U _{\tau } , w \U ) + \left( C_1 \U _y , w \U \right) + (C_0 \U , w \U ) ,   
\end{equation}   
and we estimate the different terms using integration by parts. For the term on the left we note that   
$- \left( (C_2 \U _y )_y , w \U \right) = \left( C_2 \U _y , w \U_y \right) + \left( C_2 \U_y , w_y \U \right) $,   
while $\left( C_2 \U_y , w_y \U \right) = - \left( C_2 \U , w_y \U _y \right) - \left( (w_y C_2 )_y \U, \U \right) $, so that by the symmetry of $C_2 $,   
$$   
2 | (C_2 \U_y , w_y \U ) | = | ( \U , (w_y C_2  )_y \U ) | \leq C s^2 || \ell ^{-\gamma - 1 } e^{s \eta } \U || ^2 _{L^2 (D \times (0, T^* ) } ,   
$$   
where we used that $|\partial _y ^{\alpha } w | \leq (s/\ell ) ^{\alpha } w $ for $\alpha \in \mathbb{N } . $  Since $C_2 (y) $ is continuous and positive-definite for all $y \in \Omega $ and since $D $ is compactly contained in $\Omega $, it follows that there exists a $c > 0 $ such that   
\begin{equation} \label{eq:L3.2_2}
- \left( (C_2 \U _y )_y , w \U \right) \geq c (\U_y , \U_y w ) - C s^2 \| e^{s \eta} \ell^{- \gamma -1 } \U \|^2_{L^2(D \times (0,T^* ) ) } ,
\end{equation}   
   
Turning to the terms on the right hand side of (\ref{eq:L3.2_1}), we can trivially bound the 0-th order term by a constant times $(\U , w \U ) $, while
$$
\left| \left( C_1 \U _y , w \U \right) \right| \leq C (\U , w \U )^{1/2 } (\U_y , w \U_y )^{1/2 } \leq C \varepsilon (\U_y , w \U_y ) + C \varepsilon ^{-1 }  (\U , w \U ) ,
$$
where $\varepsilon > 0 $ is arbitrary. Next, using that $\lim _{\tau \to 0, T^* } \ell ^{-2 \gamma } e^{2 s \eta } = 0 $ for any $s > 0 $, integration by parts with respect to the $\tau $-variable shows that $( \U _{\tau }, w \U) = - (\U , w \U _{\tau } ) - (\U , w_{\tau } \U ) $, and hence
$$
2 | \left( \U _{\tau } , w \U \right) | = | \left( \U , w_{\tau } \U \right) | \leq C s || \ell ^{- \gamma - 1 } e^{s \eta } \U || ^2 _{L^2 (D \times (0, T^* ) ) }  ,
$$
since $|\eta _{\tau } | \leq c \ell ^{-2 } $ and $\partial _{\tau } (\ell ^{- 2 \gamma } ) \leq c \ell ^{- 2 \gamma - 1 } . $ Finally, if we write
$$
(\f , w \U ) = ( (\ell / s ) w^{1/2 } \f , (s / \ell ) w^{1/2 } \U ) ,
$$   
and useYoung's inequality (for products), we see that $ | (\f , w \U ) | $can be bounded by $s^2 (\U , \ell ^{-2 } w \U ) + s^{-2 } (\f , \ell ^2 w \f ) . $ Combining all these estimates with (\ref{eq:L3.2_2}) and choosing $\varepsilon $ sufficiently small, (\ref{3.2a}) follows.
\medskip   
   
Next, if we differentiate the PDE with respect to $y $ we find that $\U _y $ is a solution in the weak sense of   
$$
\p_\tau \U_y - \p_y(C_2 \p_y \U_y) - \p_y \left( \widetilde{C }_1 \U_y \right) - C_0 \U_y = \f_y + (\partial _y C_0 ) \U ,   
$$   
with $\widetilde{C }_1 := C_1 + \partial _y C_2 . $ This is a system for $\U _y $ of the type (\ref{eq:gen_para_system}) but with first-order term $C_1 \partial _y \U $ replaced by $\partial _y (C_1 \U ) . $ It is clear from the proof that the bound (\ref{3.2a}) still remains valid for such a system and for weak solutions. Applying this bound with $D $ replaced by $\Delta \Subset \Delta_1 \Subset D $ we find that
\begin{eqnarray*}
& & \| e^{s \eta} \ell^{-\gamma} \U_{yy} \|^2_{L^2(\Delta \times (0,T^* ))} \\
& \leq & C \left(  s^2  \| e^{s \eta} \ell^{- \gamma - 1 } \U_y \|^2_{L^2(\Delta_1 \times (0,T^* ))} + s^{-2 }  \| e^{s \eta} \ell^{-\gamma + 1 } \f_y \|^2_{L^2(\Delta_1 \times (0,T^* ) ) } + s^{-2 } \| e^{s \eta} \ell^{-\gamma + 1 } \U \|^2_{L^2(\Delta_1 \times (0,T^* ))} \right).
\end{eqnarray*}
If we next apply \eqref{3.2a} again, with $\Delta _1 $ instead of $\Delta$ and $\gamma + 1 $ instead of  $\gamma $, we find that
\begin{eqnarray*}
& & \| e^{s \eta} \ell^{-\gamma } \U_{yy} \|^2_{L^2(\Delta \times (0,T^* ))} \\
& \leq & C \left(   s^4 \| e^{s \eta}\ell^{-(\gamma+2)} \U \|^2_{L^2(D \times (0,T^* ))} + \| e^{s \eta} \ell^{- \gamma } \f \|^2_{L^2(D \times (0,T^* ))} +  s^{-2 }  \| e^{s \eta} \ell^{-\gamma + 1 } \f_y \|^2_{L^2(D \times (0,T^* ))} \right) .
\end{eqnarray*}
Since $\U_\tau=A_2 (y) \U_{yy}+P_1(y,\p_y)\U+ \f$ with $P_1$ a first order operator, this together with (\ref{3.2a}) implies \eqref{3.2b}.   
\end{proof}

\section{Stability estimate for the linearized inverse problem}
   
Recall that $\Omega _1 = \{ y = \log (K / S^* ) : K \in I \} \Subset \Omega = \{ \log (K / S^* ) : K \in J \} . $ We will from now on assume, without essential loss of generality. that $S^* = 1 . $ Let $A_i (y) := \frac{1 }{2 } \Sigma _i (e^y ) ^2 $ for $i = 1, 2 . $ It follows from (\ref{2.6}) that $\w=\w_{\A_1}-\w_{\A_2}$ is the unique bounded solution of the linearized system
\begin{equation}\label{5.1}
\left\{
\begin{array}{llll}
\para{\p_\tau-\LL_{\A_1}}\w=\g \V,  & \textrm{in }\,\, \R\times(0,T),\cr
\w(y,0)=0,& \textrm{in }\,\,\R,\cr   
\end{array}
\right.
\end{equation}
with $G=\A_1-\A_2 $ supported on $\Omega _1 $, and $\V=(\w_{\A_2})_{yy}-(\w_{\A_2})_{y } . $ Conversely, the solution of (\ref{5.1}) can be written as such a difference, by taking $A_2  = A_1 - G . $ We want to bound $G $ in terms of $\w ( \cdot , \tau _*  ) . $ Note that $G $, although in principle matrix-valued, is diagonal, and so is given by $n $ functions, as is $\w (\cdot , \tau _*  ) . $ If we   
nterpret this as an inverse problem, it is not underdetermined.
\medskip

As a preparation, we first examine the properties of $\V (y, \tau ) $ which, as we will now explain, is related to the fundamental solution of an auxilliary parabolic system.   
Indeed, it is straightforward to verify that
\begin{equation} \label{eq:aux_PDE}
\partial _{\tau } \V- (\partial _y ^2 - \partial _y ) \left ( A_2 (y, \tau )  \V \right) - (R - Q ) \partial _y \V + ( B - Q ) \V = 0 ,
\end{equation}
with initial condition (recalling that $S^* = 1 $)
\begin{equation} \label{eq:IC_v}
\V (y, 0 ) = (\partial _y ^2 - \partial _y ) \max (1 - e^y , 0 ) e_{j^* } = \delta (y ) \, e_{j^* } .   
\end{equation}   
Hence $\V (y , \tau ) $ is the $j^* $-th column of the fundamental solution of the parabolic system (\ref{eq:aux_PDE}), and as such has the following properties.

\begin{lemma} \label{L4.1}   
Suppose that the coefficients of the matrix $A_2 $ are in $C_b ^{2, \alpha } (\mathbb{R } ) $ for some $\alpha \in (0, 1 ) . $ Then $\V \in C^{2, 1 } (\mathbb{R } \times \mathbb{R }_{> 0 } ) $ 
and that there exist for any $T > 0 $ positive constants $c $ and $C $ such that for $y \in \mathbb{R } $ and $0 < \tau \leq T $
\begin{equation} \label{eq:L4.1}
|\partial _y ^{\nu } \V (y, \tau ) | \leq \frac{C }{\tau ^{(\nu + 1 )/2 } } e^{- c |y|^2 / \tau } , \ \ \nu = 0, 1, 2 .
\end{equation}
If, moreover, $\left( e^{\tau  B } \right)_{ij^* } > 0 $ for all $i $, then all components of $\V (y, \tau ) $ remain strictly positive on compacta: there exists, for any compact subset $K \subset \mathbb{R } $ and any $\tau > 0 $, a constant $C_{\tau } > 0 $ such that      
\begin{equation}\label{5.3}
\inf_{y \in K } v_i(y, \tau ) > C_{\tau } , i = 1 , \ldots , n .
\end{equation}
\end{lemma}   
   
\noindent For the first part of the theorem, see \cite{Eidelman}, \cite{F} (in particular theorem 2 of Chapter 9 plus equation (4.19)), or \cite{EIK}, theorem 2.2. The constants will depend on the constant $\sigma _{\rm min} $ of (\ref{1.2}) and on $\| A_2 \| _{C^{2, \alpha } } $, as well as, of course, on $B $, $R $ and $Q . $ Note that since $\V $ satisfies (\ref{eq:aux_PDE}), the estimate (\ref{eq:L4.1}) implies that
\begin{equation} \label{eq:5.3a}
| \partial _{\tau } \V (y, \tau ) | \leq \frac{C }{\tau ^{3 / 2 } }e^{- c |y|^2 / \tau } .
\end{equation}
The strict lower bound (\ref{5.3}) will be proved in Appendix B; it sharpens a result of Otsuka \cite{Otsuka}. The condition $\left( e^{\tau  B } \right)_{ij^* } > 0 $ means that any state $i $ can be reached, with positive probability, from the initial state $j^* . $ The validity of this condition for arbitrary $j^* $, that is, non-vanishing of all matrix elements of $e^{\tau B } $ for arbitrary $\tau > 0 $, is equivalent to the irreducibility of the continuous time Markov chain with transition-density matrix $B . $
\medskip

\begin{remark} \rm{It is because of this lemma that we have to impose the condition that $\Sigma _2 $, and therefore $A_2 $,  is of class $C^{2, \alpha } $ in theorem \ref{T1.1}. If $A_2 $ is only $C^1 $, we can write $(\partial _y ^2 - \partial _y ) A_2  = \partial _y A_2 \partial _y - \partial _y \left( A_2 + [A_2 , \partial _y ] \right) $, and interpret $\V $ as the weak solution of a parabolic system of the form $\partial _{\tau } \V -  \partial _y A_2 \partial _y \V + \partial _y (E \V ) + F \V = 0 $, with $E $ continuous, but we do not know whether the Gaussian estimates for in particular the derivatives of $\V $ still remain valid then.      
}   
\end{remark}   
   
The following lemma is the main technical step in the proof of our stability estimate. Recall that, by hypothesis, $G (y) = 0 $ outside of $\Omega _1 . $ Let $\chi \in C^{\infty } _c (\Omega ) $ be equal to 1 on a neighborhood of $\overline{\Omega }_1 $, and let $\omega \subset \Omega \setminus \overline{\Omega }_1 $ be an arbitrary neighborhood of ${\rm supp } (\chi _y ) . $ Note that $\omega $ will be disconnected, but that we can, and will, arrange for $\omega $ to have exactly two components, $\omega _1 $ and $\omega _2 . $      

\begin{lemma}\label{L5.2} There exists a constant $C>0 $, only depending on $|| A_1 ||_{C^1 (\mathbb{R } ) } $, $|| A_2 ||_{C^{2, \alpha } (\mathbb{R } ) } $, $B $, $\Omega _1 $, $\Omega _2 $, $\tau ^* $ and $n $, such that for any $G \in L^2 (\mathbb{R } ) ^n $ supported in $\Omega _1 $, the solution $\w $ of (\ref{5.1}) satisfies the estimate   
\begin{equation}\label{5.4}
\norm{G}^2_{L^2(\Omega_1)}
\leq C \left( \norm{\w(\cdot,\tau _* )}^2_{H^2(\Omega)} + \| \w \|_{L^2(\omega \times (0,T))}^2 \right) .   
\end{equation}   
\end{lemma}   
   
\begin{proof}   
Since $\chi G = G $ and since $G $ and $\w $ are related by (\ref{5.1}), we have   
\begin{eqnarray} \nonumber
\| e^{s \eta (y , \tau _*  ) } G(y) \V (y, \tau _*  ) \|_{L^2 (\mathbb{R } ) } ^2 &=&\| e^{s \eta (y , \tau _*  ) } \chi (y) G(y) \V (y, \tau _*  ) \|_{L^2 (\mathbb{R } ) } ^2 \\
&\leq & 2 \| e^{s \eta (y , \tau _*  ) } \chi (y) \w _{\tau } (y , \tau _*  ) \|_{L^2 (\mathbb{R } ) } ^2 + 2 \| \chi (y) \mathcal{L }_{A_1 } \w (y, \tau _*  ) \| _{L^2 (\mathbb{R } ) } ^2 \nonumber \\
&\leq & 2 \| e^{s \eta (y , \tau _*  ) } \chi (y) \w _{\tau } (y , \tau _*  ) \|_{L^2 (\mathbb{R } ) } ^2 + C \| \w (y , \tau _*  ) \|_{H^2 (\Omega ) } ^2 , \label{eq:alt_proof_1}
\end{eqnarray}
with $C $ depending on $|| A_1 ||_{\infty } $ and on $|| R || $, $|| B || $ and $||Q || . $ We next seek to control the first term on the right by applying the Carleman estimate of Theorem \ref{L3.1} to the function $\z = \chi \w_\tau $, which satisfies the PDE   
\begin{equation} \label{eq:alt_proof_2}
\z _{\tau } - \mathcal{L }_{A_1 } \z = G \V _{\tau } + \mathcal{Q }_1 \w _{\tau } ,
\end{equation}
where $\mathcal{Q}_1 := [ \chi , \LL _{\A_1 } ] $ is a first order partial differential operator in $y $ with coefficients supported in $K := {\rm supp } (\chi _y ) \subset \omega . $ Since $\lim_{\tau \downarrow 0} e^{2s\eta(y,\tau)} = 0$, the first term on the right hand side of (\ref{eq:alt_proof_1}) can be bounded by
\begin{eqnarray*}
\| e^{s \eta (y , \tau _*  ) } \z (y, \tau _*  ) \| _2 ^2 &=& \int _0 ^{\tau _*  } \partial _{\tau } || e^{s \eta } \z ||_2 ^2 d \tau = 2 \int _0 ^{\tau _*  } (s\eta _{\tau } \z + \z_{\tau } , \z ) _{L^2 (\mathbb{R } ) } \, e^{2 s \eta (y, \tau ) } d\tau \\
&\leq & C s^{-1 } \left( \| \frac{s }{\ell } \z e^{s \eta } ||_{L^2 (\Omega \times (0, \tau _*  ) ) } ^2 + | ( s^{-1/2 } \z _{\tau } , s^{3/2 } e^{2 s \eta } \z )_{L^2 (\Omega \times (0, \tau _*  ) ) } | \right) ,
\end{eqnarray*}
where we used that $|\eta _{\tau } (y, \tau ) | \leq C / \ell (\tau ) ^2 . $ Since $\ell (\tau ) $ is bounded on $(0, T^* ) $,   
\begin{eqnarray*}
| ( s^{-1/2 } \z _{\tau } , s^{3/2 } e^{2 s \eta } \z )_{L^2 (\Omega \times (0, \tau _*  ) ) }  | &\leq & C | \left( (s /\ell )^{-1/2 } \z _{\tau } , (s / \ell )^{3/2 } e^{2 s \eta } \z \right) _{L^2 (\Omega \times (0, \tau _*  ) ) } \\
&\leq & C \left( \| (s/\ell )^{-1/2 } e^{s \eta } \z _{\tau } \| _{L^2 (\Omega \times (0, \tau _*  )  ) } ^2 + \| (s/\ell )^{3/2 } e^{s \eta } \z \| _{L^2 (\Omega \times (0, \tau _*  ) ) } \right) .   
\end{eqnarray*}
Replacing the $L^2 $-norms in the last line by those of $L^2 (Q ) $, the Carleman estimate (\ref{eq:Carleman_est}) and equation (\ref{eq:alt_proof_2}) then
imply that for sufficiently large $s $,
$$
\| e^{s \eta (y , \tau _*  ) } \z (y, \tau _*  ) \| _{L^2 (Q) } ^2 \leq C s^{-1 } \left( || e^{s \eta } G \V _{\tau } ||_{L^2 (Q) } ^2 + \|
e^{s\eta} \w_\tau \|_{L^2(K \times (0,T^* ) ) }^2 + \| e^{s\eta} \w_{\tau y} \|_{L^2(K \times (0,T^* ) ) }^2 \right) ,      
$$
since $\mathcal{Q }_1 $ is a first order partial differential operator supported on $K . $ Inserting this into (\ref{eq:alt_proof_1}) we find that
\begin{eqnarray}
\| e^{s\eta (y, \tau _* )} G (y) \V (y, \tau _* ) \|_{L^2(\Omega_1)}^2
& \leq & C s^{-1 } \left(  \| e^{s\eta} G \V _{\tau }  \|_{L^2(\Omega_1 \times (0,T^* ) ) }^2 + \| e^{s\eta} \w_\tau  \|_{L^2(K \times (0,T^* ) ) }^2 \right. \nonumber \\
& & \hspace*{.5cm} + \left. \|e^{s\eta} \w_{\tau y} \|_{L^2(K \times (0,T^* ) ) }^2 \right) + C \| \w (y, \tau _*  ) \| _{H^2 (\Omega ) } ^2 . \label{5.9}
\end{eqnarray}
We next show that the first term on the right can be absorbed into the left hand side if $s $ is sufficiently large. We claim that there exists a constant $C > 0 $ such that for all $y \in \Omega $,
\begin{equation} \label{eq:5.9.5}
\sup _{\tau \in (0, T^* ) } |\V _{\tau } (y, \tau ) |^2 e^{2 s \eta (y, \tau ) } \leq C e^{2 s \eta (y, \tau _*  ) } .     
\end{equation}
This is basically a consequence of the fact that, by construction, $\eta (y, \tau ) $ has its maximum on $(0, T^* ) $ in $\tau = T^* / 2 = \tau _*  $ but we have to take some care with the singularity of $\V _{\tau } (\cdot , \tau ) $ in $\tau = 0 . $ Write $\eta (y, \tau ) = - \widetilde{\eta }(y) / \ell (\tau ) $, with $\widetilde{\eta }(y) = e^{\lambda L } - e^{\lambda \psi } $ and let $0 < a < \tau _*  . $ Since $\ell (\tau ) ^{-1 } $ assumes its absolute minimum on $(0, T ) $ in $\tau = \tau _*  $, there exists an $\varepsilon > 0 $ such that $(1 - \varepsilon ) \inf _{(0, a ] } \ell (\tau ) ^{-1 } > \ell (\tau _*  ) ^{-1 } . $ By (\ref{eq:5.3a}), $|\V _{\tau } (y, \tau ) | \leq C \tau ^{- 3 / 2 } $ on $\Omega . $ Since $\inf _{\Omega } \widetilde{\eta } > 0 $ , there exists a constant $C_{\varepsilon } $ such that $\tau ^{- 3 } \leq C_{\varepsilon } e^{ 2 s \varepsilon \widetilde{\eta } / \ell (\tau ) } $ for $s \geq s_0 $, and therefore   
$$
\max _{\tau \in (0, a ] } |\V_{\tau } (y, \tau ) |^2 e^{2 s \eta (y, \tau ) } \leq C_{\varepsilon } \max _{\tau \in (0, a ] } e^{ - 2 s (1 - \varepsilon ) \widetilde{\eta }(y) / \ell (\tau )  } \leq C_{\varepsilon } e^{ - 2 s \widetilde{\eta }(y) / \ell (\tau _*  ) } = C_{\varepsilon } e^{2 s \eta (y, \tau _*  ) } .
$$
Since $|\V _y (y, \tau ) | $ is bounded on $[a, T^* ] $, a similar estimate holds on $[a, T^* ] $, and (\ref{eq:5.9.5}) follows.
\medskip

As a consequence, we have that $ s^{-1 } \| e^{s\eta} G \V _{\tau } \|_{L^2(\Omega_1 \times (0,T^* ) ) } \leq C s^{-1 } T^* \| e^{s\eta(\cdot,\tau _* )} G \|_{L^2(\Omega_1)}$. Hence by (\ref{5.3}),  (\ref{5.9}) implies that for sufficiently large $s $,
\begin{equation}
\|  e^{s\eta(\cdot,\tau _* )} G \|_{L^2(\Omega_1)}^2 \leq C \| \w (y, \tau _*  ) \| _{H^2 (\Omega ) } ^2 + C s^{-1 } \left( \|  e^{s\eta} \w_\tau \|_{L^2(K \times (0,T^* ) )} ^2 + \| e^{s\eta} \w_{\tau y} \|_{L^2(K \times (0,T^* ) ) }^2 \right) . \label{5.9b}
\end{equation}
To finish the proof, we use Lemma \ref{L3.2} to show that the last two terms on the right can be bounded by $\| \w \| _{L^2 (\omega \times (0, T^* ) ) } $ where we recall that $\omega \subset \Omega \setminus \Omega _1 $ is an arbitrarily small neighborhood of $K . $ Lemma \ref{L3.2} applies since $A_1 $ and $G $ are both $C^1 $, by assumption. Choose $\omega ' $ such that $K \subset \omega ' \Subset \omega . $ By \eqref{3.2a} with $\U=\w_\tau$, $\gamma=0 $, $\Delta = K$ and $D=\omega ' $, we find that
$$
\| e^{s \eta} \w_{\tau y} \|_{L^2(K \times (0,T^* ) ) } \leq Cs \| e^{s \eta} \ell^{-1} \w_{\tau} \|_{L^2(\omega ' \times (0,T^* ) ) } ,   
$$
since $(\partial _{\tau } - \mathcal{L }_{A_1 } ) \w _{\tau } =  G \V _{\tau } $ vanishes on $\omega ' $. Applying \eqref{3.2b} with $\U=\w$, $\gamma  = 0 $ or 1 and $\Delta=\omega ' $ and $D=\omega $ we then see (again using the vanishing of $G $, on $\omega $ this time), that
$$
 \|  e^{s\eta} \w_\tau \|_{L^2(K \times (0,T))}^2 + \| e^{s\eta} \w_{\tau y} \|_{L^2(K \times (0,T))}^2 \leq s^6 \| e^{s \eta} \ell^{ - 3 } \w \|_{L^2(\omega \times (0,T))}^2 ,
$$
which is bounded by $C || \w ||_{L^2 (\omega \times (0, T^* ) ) } $ for any $s \geq 1 $, since $\sup_{(y,\tau)\in \omega \times(0,T^* )} s^3 \ell^{-3 }(\tau) e^{s\eta(y,\tau)} < \infty $, for any $s . $ Since $e^{s \eta (\cdot , \tau _*  ) } $ is bounded from below by a strictly positive constant on  $\Omega _1 $, the lemma follows.
\end{proof}

If we do not suppose that $G $ has compact support, we can generalise Lemma \ref{L5.2} as follows:   

\begin{corollary} \label{cor:L5.2}   
Let $\Omega _1 \Subset \Omega \subset \mathbb{R } $ be bounded open intervals and $\omega \Subset \Omega \setminus \overline{\Omega }_1 $ an open subset such that $\omega $ has non-zero intersection with each component of $\Omega \setminus \Omega _1 $ and such that, moreover, $0 \notin \overline{\omega } . $ Then there exists a constant $C > 0$ such that if $\w $ solves (\ref{5.1}) on $\Omega $, then   
\begin{equation} \label{eq:5.10}
|| G ||_{L^2 (\Omega _1 ) } ^2 \leq C \left( ||\w (\cdot , \tau _*  ) ||_{H^2 (\Omega ) } ^2 + || \w ||_{L^2 (\omega \times (0, T^* ) ) } ^2 + 
|| G ||_{H^1 (\omega ) } ^2 \right) .   
\end{equation}
\end{corollary}   
   
\begin{proof} Let $\chi \in C^{\infty } _c (\mathbb{R } ) $ such that $\chi = 1 $ on a neighborhood of $\overline{\Omega }_1 $, while ${\rm supp }(\chi _y ) \subset \omega . $ Then the proof of (\ref{5.9b}) shows that
\begin{equation} \nonumber
\| e^{s \eta (y, \tau _*  ) } \chi G \| _{L^2 (\mathbb{R } ) } ^2 \leq C \| \w \| _{H^2 ({\rm supp } (\chi ) ) } ^2 + C s^{-1 } \left( \|  e^{s\eta} \w_\tau \|_{L^2({\rm supp } (\chi _y ) \times (0,T^* ) ) }^2 + \| e^{s\eta} \w_{\tau y} \|_{L^2 (\rm{supp } (\chi _y ) \times (0,T^* ) ) }^2 \right) .
\end{equation}
An application of lemma \ref{L3.2} then shows that since $\omega $ is a bounded open neighborhood of the ${\rm supp }(\chi _y ) $, then
\begin{eqnarray*}
\| e^{s \eta (y, \tau _*  ) } \chi G \| _{L^2 (\mathbb{R } ) }^2 &\leq & C \left( \| \w \| _{H^2 ({\rm supp } (\chi ) ) } ^2 + s^6 \| e^{s \eta } \ell ^{-3 } \w \| _{L^2 (\omega \times (0, T^* ) ) } ^2 \right. \\
&& \left. + s \| e^{s \eta } \ell ^{-1 } G \V _{\tau } \| _{L^2 (\omega \times (0, T^* ) ) } ^2 + s^{-1 } \| e^{s \eta } (G \V )_y \| _{L^2 (\omega \times (0, T^* ) ) } ^2 \right) ,   \nonumber
\end{eqnarray*}
where the norms of $G $ and $G_y $ on the right occur since $G $ is no longer 0 on a neighborhood of ${\rm supp }(\chi _y ) . $ The corollary follows by observing that $\V $, $\V _{\tau } $ and $\V _y $ are bounded on $\omega \times (0, T^* $ ), since $0 \notin \overline{\omega } $, and that $\chi = 1 $ on $\overline{\Omega }_1 . $   
\end{proof}   
   
We observe in passing that the $H^1 (\omega ) $-norm of $G $ in (\ref{eq:5.10}) can be replaced by $C_{\varepsilon } || G ||_{L^2 (\omega ) } ^2 +  \varepsilon || G_y ||_{L^2 (\omega ) } ^2 || $ for arbitrary $\varepsilon > 0 $, by choosing $s $ sufficiently large.

\section{Completion of the proof of Theorem \ref{T1.1}}
\label{sec:Comp}
The final step in the proof is to get rid of the last term on the right hand side of (\ref{5.4}). We proceed in a succession of lemmas. Inspired by \cite{SU} we start by establishing   
   
\begin{lemma}\label{L6.1}   
Let $X $, $Y_1 \subset Y $ and $Z $ be four Banach spaces such that the inclusion $id : Y_1 \to Y $ is continuous. Let $\mathcal{A} : X \to Y$ be a bounded injective linear operator, $\mathcal{K}: X \to Z$ be a compact linear operator and let $X_1 \subset X $ be a linear subspace which is mapped into $Y_1 $ by $\mathcal{A } $ and for which there exists a constant $C>0$ such that      
\begin{equation}\label{6.1}
\norm{f}_X \leq C \norm{\mathcal{A}f}_{Y_1 } +\norm{\mathcal{K}f}_Z ,   
\end{equation}   
for all $f \in X_1 . $ Then there exists a (in general different) constant $C > 0 $ such that for all $f \in X_1 $ we have that   
\begin{equation}\label{6.2}
\norm{f}_X \leq C \norm{\mathcal{A}f}_{Y_1 } .
\end{equation}
\end{lemma}
\begin{proof}
We argue by contradiction. Assume that (\ref{6.2}) does not hold. Then there exists a sequence $(f_n)_n$ in $X_1 $ such that $\norm{f_n}_X = 1$ for all $n$ and $|| \mathcal{A}f_n ||_{Y_1 } \to 0 $ as $n$ goes to infinity. Since $\mathcal{K}: X \to Z$ is compact, there is  a subsequence, still denoted by $f_n$, such that $(\mathcal{K}f_n)_n$ converges in $Z . $ Hence this is a Cauchy sequence in $Z $ and, by applying (\ref{6.1}) to $f_n-f_m$, we get that $\norm{f_n-f_m}_X \to 0$, as $n, m \to \infty $; $(f_n)_n$ is a therefore Cauchy sequence in $X$ and $f_n \to f$ in $X $ as $n\to\infty$ for some $f \in X$. Since $\norm{f_n}_X = 1$ for all $n $, it follows that $\norm{f}_X = 1 $ also.   
Since $\mathcal{A } $ is continuous, we have $\mathcal{A } f_n \to \mathcal{A } f $ in $Y . $ On the other hand, $\mathcal{A } f_n \to 0 $ in $Y $ since this is true in $Y_1 . $ Hence $\mathcal{A } f = 0 $, which is a contradiction with $\mathcal{A}$ being injective.
\end{proof}   
   
In the light of (\ref{5.4}), and identifying the diagonal matrix $G $ with the vector of its diagonal elements, we will apply this lemma with the Banach spaces $X=\set{ G \in L^2(\R) ^n ,\ {\rm supp }(G) \subset \Omega _1 } $, $Y_1 = H^2(\Omega) ^n \hookrightarrow Y = H^1 (\Omega ) ^n $ and $Z= L^2( \omega \times (0,T) ) ^n $, and with the operators $\mathcal{A } $ and $\mathcal{K } $ defined by   
$$
\mathcal{A}: X\to Y,\ \mathcal{A} G =\w(\cdot,\tau _* ) |_{\Omega } ,   
$$   
and   
$$   
\mathcal{K} : X\to Z,\ \mathcal{K} G=\w | _{\, \omega \times (0,T^* ) } ,
$$
where $\w $ denotes the unique solution to (\ref{5.1}). The following lemma will imply that $\mathcal{A  } $ is well-defined; here and below, $C $ denotes a generic constant which is independent of $\tau _* . $   
   
\begin{lemma} \label{L6.2} For $G \in L^2 (\mathbb{R } )  ^n $, let $\mathcal{A }_1 G $ be the function $\w (\cdot , \tau _* ) $ on all of $\mathbb{R } . $ Then $\mathcal{A }_1 : L^2 (\mathbb{R } )^n \to H^1 (\mathbb{R } )^n $ is bounded with norm majorized by $C \max (1, \sqrt{\tau _* } ) . $   Moreover, if $G \in H^1 (\mathbb{R } )^n $, then $\mathcal{A }_1 G \in H^2 (\mathbb{R } ) ^n $ with $H^2 $-norm bounded by $C \sqrt{ \max (\tau _* , \tau _* ^{-1 } ) } \, 
|| G ||_{H^1 (\mathbb{R } ) } . $   
\end{lemma}   
   
\noindent In particular, since $\mathcal{A } G = \mathcal{A }_1 G |_{\Omega } $ for $G \in X $, $\mathcal{A } : X \to Y $ is continuous, and $\mathcal{A } G \in Y_1 $ if $G \in X $ is $C^1 $, which is true for $G = A_1 - A_2 . $      
\medskip   
   
The somewhat technical proof of this lemma is given in Appendix C. It uses Schur's bound for the $L^2 $-norm of an integral operator. Here we just note that, because of the singular behavior of $\V $ in the right hand side of (\ref{5.1}), the boundedness of $\mathcal{A } $ in $L^2 $-norm does not follow directly from the standard energy estimates for parabolic equations, unless one would for example assume that $0 \notin {\rm supp }(G ) $,    
\begin{lemma} \label{L6.2a}
The operator $\mathcal{A } $ is injective.
\end{lemma}

\begin{proof}
Since  $\w $ is solution of \eqref{5.1}, we deduce from the identities $\w(\cdot, \tau _* ) = 0$ and $G = 0$ on $\omega \subset\Omega \backslash\Omega_1$ that $\w_\tau(\cdot,\tau _* )=0$ on $\omega$. Arguing in the same way we get that the successive derivatives of $\w $ with respect to $\tau $ vanish on $\omega \times\set{\tau _* }$. Since $\w $ is the difference of two solutions to an initial value problem with time independent coefficients,  $\w $ is analytic in $\tau $, so we necessarily have $\w= 0$ on $\omega \times(0,\tau _* )$. Therefore $G = 0 $ on $\Omega_1$ by (\ref{5.4}), and the proof is complete.
\end{proof}   
   
To prove compacticity of $\mathcal{K } $ we will use the following classical parabolic regularity estimate: cf. \cite{Evans}[\S 7.1, Thm 5]; the lemma is stated there for scalar parabolic equatons, but the proof remains valid for parabolic systems.

\begin{lemma}\label{L6.2b} Let $I\subset\R$ be a bounded interval, let $F\in L^2(I \times (0,T^* ) ) $ and let $\U \in H^{2,1}(I \times (0,T^* ) ) $ be solution to
$$
\left\{\begin{array}{lll}
\para{\p_\tau-\mathcal{L }_{A_1 } } \U = F, & y\in I,\ \tau \in (0,T^* ) \cr
\U (y, 0 ) = 0,&  y\in I,\cr
\U (y, \tau  ) = 0 & y\in\p I,\ \tau \in (0,T^* ).
\end{array}
\right.
$$
Then there exists a positive constant $C $ such that   
$$
\norm{\U }_{H^{2,1}(I \times (0,T^* ) ) } \leq C\norm{F}_{L^2(I \times (0,T^* ) ) } .
$$
\end{lemma}

\begin{lemma}\label{L6.3}
$\mathcal{K} : X \to Z $ is a compact operator.
\end{lemma}

\begin{proof} Recall that $\omega = \omega _1 \cup \omega _2 $ with $\omega _1 $ and $\omega _2 $ open and disjoint. Let $I_j \subset \mathbb{R } \setminus \Omega _1 $ be an open interval such that $\overline{\omega}_j \subset I_j $, and let $\chi _j \in C^{\infty } _c (I ) $ be such that $\chi _j = 1 $ on $\omega _j $ for $j = 1, 2 . $ Then since $G = 0 $ on $I_j $, 
the function $\U _j := \chi _j \w $ satisfies the boundary value problem of lemma \ref{L6.2b} with $F = [ \chi _j , \mathcal{L }_{A_1 } ] \w . $ Hence
$$
\norm{\w }_{H^{2,1}(\omega _j \times (0,T ) ) } \leq  \norm{\U _j }_{H^{2,1}(I \times (0,T ) ) } \leq C \left( || \w ||_{L^2 (I _j \times (0, T ) ) } + ||\w _y ||_{L^2 (I _j \times (0, T ) ) } \right) .
$$
If we apply the first statement of lemma \ref{L6.2} and integrate over $\tau _* $ from 0 to $T^* $, we find that the right hand side is bounded by $C \max (\sqrt{T^* }, T^* ) ||G ||_{L^2 (\mathbb{R } ) } $, and the lemma follows from the compactness of the injection $H^1(\omega \times (0,T ) ) ^n \hookrightarrow Z=(L^2(\omega \times (0,T ) ) ^n . $
\end{proof}

Note that we needed lemma \ref{L6.2b} to get a suitable bound for the $L^2 $-norm of $\w _{\tau } $: those for $\w $ and $\w _y $ already follow from lemma \ref{L6.2}.

\begin{remark} \label{remark:L6.3} \rm{The operator $\mathcal{K } $ extends naturally to an operator $\mathcal{K } : L^2 (\mathbb{R } )^n\to Z = L^2 (\omega \times (0, T^* ) )^n . $ We claim that if $0 \notin \overline{\omega } $ then $\mathcal{K } $ is still compact, where $\omega $ may be an otherwise arbitrary bounded open subset of $\mathbb{R } . $ To prove this, we may assume that $\omega $ is connected. Let $I $ an open interval containing $\overline{\omega } $ but not containing 0, and let $\chi $ and $\U $ be as above. Since $G $ is no longer 0 on $I $, we now find that $\U $ is a solution of the system of lemma \ref{L6.2b} with right hand side $F = [ \chi , \mathcal{L }_{A _1 } ] \w + \chi G \V . $ But $\chi G \V \in L^2 (I \times (0, T ) )^n $ since $0 \notin {\rm supp }(\chi ) $ and it follows as before that $\mathcal{K } $ sends $L^2 (\mathbb{R } )^n $ into $H^{2, 1 } (\omega \times (0, T ) ) . $    
}   
\end{remark}   
   
If we finally combine Lemmas \ref{L5.2}, \ref{L6.1}, \ref{L6.2} and \ref{L6.3}, we conclude that if $G \in H^1 (\mathbb{R } ) ^n $ is supported in $\Omega $, then   
\begin{equation} \label{eq:7.1}
\norm{G}_{L^2(\Omega_1)}\leq C\norm{\mathcal{A} G}_Y=C\norm{\w(\cdot,\tau _* )}_{H^2(\Omega)} .   
\end{equation}   
This applies in particular to $G = A_1-A_2 $, which therefore proves Theorem \ref{T1.1}.

\section{Extension to non-compactly supported $G $'s   }   
   
We can extend (\ref{eq:7.1})  to non-compactly supported $G $ by allowing suitable norms of $G $ on $\mathbb{R } \setminus \Omega _1 $ on the right hand side of the inequality. An easy extension is the following:   
   
\begin{corollary}      
For any pair of bounded open intervals $\Omega _1 \Subset \Omega $ and any $\tau _*  > 0 $, there exist a constant $C > 0 $, depending on these subsets and on $\tau ^* $ as well as on $|| A_1 ||_{C^1 (\mathbb{R } ) } $ and $||A_2 ||_{C^{2, \alpha } (\mathbb{R } ) } $, such that for all $G \in H^1 (\mathbb{R } )^n $, 
\begin{equation} \label{eq:7.2}
||G ||_{L^2 (\Omega _1 ) } \leq C \left( || \w (\cdot , \tau _*  ) ||_{H^2 (\Omega ) } + || G ||_{H^1 (\mathbb{R } \setminus \Omega _1 ) } \right) .
\end{equation}   
\end{corollary}   
   
\begin{proof} Let $\chi \in C^{\infty } _c (\Omega ) $ such that $\chi (y) = 1 $ on $\Omega _1 . $ Write $G = G_1 + G_2 $ with $G_1 := \chi G $ and $G_2 := (1 - \chi ) G . $ Let $\w _j $, $j = 1, 2 $, be the solution of the boundary value problem (\ref{5.1}) with right hand side $G_j \V $ and $\w $ the one with right hand side $G \V . $ Then $\w = \w _1 + \w _2 $, by uniqueness, and by (\ref{eq:7.1}) above,   
$$
|| G_1 ||_{L^2 (\Omega _1 ) } \leq C || \w_1 (\cdot , \tau _*  ) ||_{H^1 (\Omega ) } \leq C \left( || \w ( \cdot , \tau _* ) ||_{H^1 (\Omega ) } + || \w _2 (\cdot , \tau _* ) || _{H^2 (\Omega ) } \right) ,
$$
while by the second statement of lemma \ref{L6.2}, $|| \w _2 (\cdot , \tau _* ) ||_{H^1 (\mathbb{R } ) } \leq C || G_2 ||_{H^1 (\mathbb{R } ) } \leq C || G ||_{H^1 (\mathbb{R } \setminus \Omega _1 ) } . $ Inequality (\ref{eq:7.2}) follows.
\end{proof}
\medskip
   
Going back to our original problem, the corollary implies that if we take an $\varepsilon > 0 $ and $I \Subset J $ and if $\Sigma _1 - \Sigma _2  $ has a sufficiently small $H^1 $-norm on the complement of $I $, 
then we can achieve "stability up to $\varepsilon $":
\begin{equation} \label{eq:7.4}
\norm{\Sigma_1 - \Sigma_2}_{L^2(I) } ^2 \leq C \sum _{i = 1 } ^n \ \norm{C^*_{\Sigma_1 } (\cdot , \epsilon _i , T ) - C^*_{\Sigma_2 } (\cdot , \epsilon _i , T ) }_{H^2(J ) } ^2 + \varepsilon .
\end{equation}   
If we would be content to work within a certain numerical precision $\varepsilon $, it would therefore not be necessary that $\Sigma _1 = \Sigma _2 $ outside of $I $, only that it has a sufficiently small $H^1 $-norm there.
\medskip

We can improve the corollary by re-examining the proof of (\ref{eq:7.1}).

\begin{theorem} \label{thm:7.1}   
Let $\Omega _1 \Subset \Omega $ be two bounded intervals and let $\omega \Subset \Omega \setminus \Omega _1 $ be an open subset with $0 \notin \omega $, such that $\omega $ intersects each of the two components of $\Omega \setminus \Omega _1 $ in an open interval. Finally, let $\tau _*  > 0 . $ Then   
there exists a constant $C $ such that   
\begin{equation} \label{eq:thm7.1}
|| G ||_{L^2 (\Omega _1 ) } \leq C \left( || \w (\cdot , \tau _*  ) ||_{H^2 (\Omega ) } + || G ||_{H^1 (\omega ) } + || G ||_{L^2 (\mathbb{R } \setminus \Omega _1 ) } \right) .   
\end{equation}
\end{theorem}

\begin{proof} By contradiction: suppose (\ref{eq:thm7.1}) does not hold. Then there exists a sequence $(G_n )_n $ such that $|| G_n ||_{L^2 (\Omega _1 ) } = 1 $ and
such that
$$
||\w_n (\cdot , \tau _*  ) ||_{H^2 (\Omega ) } , || G_n ||_{H^1 (\omega ) } ,|| G_n ||_{L^2 (\mathbb{R } \setminus \Omega _1 ) } \to 0 , \ \ n \to \infty ,
$$
where $w_n $ is the solution of (\ref{5.1}) with right hand side $G_n \V . $ By corollary \ref{cor:L5.2},
\begin{equation} \label{eq:7.5}
|| G_n ||_{L^2 (\Omega _1 ) } \leq C \left( ||\w_n (\cdot , \tau _*  ) ||_{H^2 (\Omega ) } + ||\w_n ||_{L^2 (\omega \times (0, T ) ) } + || G_n ||_{H^1 (\omega ) } \right) .   \end{equation}
The sequence $(G_n )_n $ is clearly norm-bounded in $L^ 2 (\mathbb{R } ) $, so by remark \ref{remark:L6.3} there exists a subsequence, still denoted by $(G_n )_n $ such that $(\w _n )_n $ converges in $L^2 (\omega \times (0, \tau _*  ) ) . $ It then follows from (\ref{eq:7.5}) that $G_n | _{\Omega _1 } $ is a Cauchy sequence in $L^2 (\Omega _1 ) ^n $, which therefore converges to a function $G $ on $\Omega _1 . $ Extending $G $ by 0 on the complement of $\Omega _1 $ and recalling that  $|| G_n ||_{L^2 (\mathbb{R } \setminus \Omega _1 ) } \to 0 $, we see that $G_n \to G $ in $L^2 (\mathbb{R } )^n . $ Clearly, $||G ||_{L^2 (\mathbb{R } ) } = 1 . $ If $\w $ is the solution of (\ref{5.1}) with right hand side $G \V $, then by lemma \ref{L6.2}, $\w_n (\cdot , \tau _*  ) \to \w (\cdot , \tau _*  ) $ in $H^1 . $ Hence $w(\cdot , \tau _*  ) = 0 $ on $\Omega _1 $, while ${\rm supp } (G ) \subset \Omega _1 $ and $||G ||_{L^2 (\Omega _1 ) } = 1 . $ But this gives a contradiction with lemma \ref{L6.2a}.

\end{proof}   

We can always choose $\omega $ such that $0 \notin \omega . $ Theorem \ref{thm:1.2} is an immediate corollary of theorem \ref{thm:7.1}, on using the trivial bound $|| G ||_{H^1 (\Omega ) } \leq || G ||_{H^1 (\Omega _1 ) } . $ One can slightly strengthen (\ref{eq:thm1.2a}) by replacing the $H^1 $-norm of $\Sigma _1 - \Sigma _2 $ on $J \setminus I $ by one on a compactly contained open subset which intersects each component of $J \setminus I $, but we opted for the simpler formulation.   
\medskip   
   
\appendix   
   
\section{Proof of theorem \ref{thm:RS_Dupire} }

Let $E(K, T ; S , t ) = \left( e_{ij } (K, T ; S , t ) \right)_{i, j } $ be the fundamental solution of the system (\ref{eq:RSLV_BSE6a}):   
\begin{equation} \label{eq:FS_BW}
\partial _t e_{ij }  + \frac{1 }{2 } \sigma _j (S, t )^2 S^2 \partial _S ^2 e_{ij }  + (r_j - q_j ) S \partial _S e_{ij } + \sum _k b_{kj } e_{ik } = r_j e_{ij } , 
\ \ t \leq T ,
\end{equation}   
with final boundary condition $e_{ij } (K, T ; S, T ) = \delta (S - K ) \delta _{ij } . $ The existence of such a fundamental solution follows from standard results for parabolic systems after a logarithmic change of variables $S = e^x $: see \cite{Eidelman}, \cite{F} and also Appendix B below. It is known that the $e_{ij } $ are $C^2 $ in $S $ and $K $, and $C^1 $ in $t $ and $T $, for $t < T . $    
We can express the generalized call prices in terms of this fundamental solution as   
\begin{equation} \label{eq:call_price_repr}
\underline{C } (K , T , \pi ; S , t ) = \int _{\mathbb{R }_{> 0 } } \,  \underline{\pi }^t E (y, T ; S , t ) \max (y - K , 0 ) \, dy ,
\end{equation}
or, component-wise,   
$$
C_j (K, \pi , T ; S , t ) = \int _K ^{\infty } \left( \sum _i \pi _i e_{i j } (y , T ; S , t ) \right) (y - K ) \, dy, \ \ j = 1 , \dots , n ,   
$$   
where the integral converges absolutely.    
Differentiating twice with respect to $K $ yields the extension to regime-switching models of the well-known Breeden-Litzenberger formula,
\begin{equation} \label{eq:BL}
\partial _K ^2 C_j (K , T , \pi , S , t ) = \sum _i \pi _i e_{ij } (K , T ; S , t ) ,
\end{equation}
or, in vector notation,  $\partial _K ^2 \underline{C } (K, \pi , T ; S , t ) = \underline{\pi }^t E(K, T ; S , t ) . $   
\medskip   

Under our assumption of market completeness we dispose of call options with $\pi = \epsilon _i $ for $i = 1 , \ldots , n $, and the Breedenberg-Litzenberg relation (\ref{eq:BL}) then allows, at a time $t^* $, to recover the $j $-th column of $P (K , T ; S_{t^* } , t^* ) $, where $e_j = X_{t^* } . $
\medskip

In matrix notation, (\ref{eq:FS_BW}) reads   
\begin{equation} \label{eq:FS2}
\partial _t + S^2 (\partial _S ^2 E) \Sigma + S (\partial _S E ) (R - Q ) + E (B - R ) = 0 , \ \ t < T .
\end{equation}
The usual argument for deriving a PDE in the forward variables $(K, T ) $ for $E $ also apply to the system case: we recall the derivation for convenience of the reader while paying attention to the regularity needed for $\Sigma . $ Let the row-vector $\underline{V } (S, t ) $ be a solution of (\ref{eq:RSLV_BSE2}) with arbitrary final value $\underline{F } \in C^{\infty } _c \left( \mathbb{R }_{> 0 } \right) ^n . $ If $t < u < T $, then, by uniqueness of solution,   
$$
\underline{V }(S, t ) = \int _{\mathbb{R }_{> 0 } } \, \underline{V }(y, u ) E (y, u ; S , t ) \, dy .
$$   
If we differentiate this relation with respect to $u $, use (\ref{eq:RSLV_BSE2}) for $\underline{V } $ and, assuming momentarily that $\Sigma ^2 $ is $C^2 $, integrate by parts twice, we find that
$$
0 = \int _{\mathbb{R }_{> 0 } } \underline{V }(y, u ) ^t \left( \partial _u E - \frac{1 }{2 } \partial _y ^2 \left( y^2 \Sigma ^2 E \right) + S \partial _y (y (R - Q ) E ) + (R - B ) E \right) dy ,
$$
where $E = E (y , u ; S, t ) . $ Letting $u \to T $ and using that $\underline{F } $ is arbitrary, we find the forward equation   
\begin{equation} \label{FS_FW}
\partial _T E - \frac{1 }{2 } \partial _y ^2 \left( y^2 \Sigma (y, T ) ^2 E \right) + (R - Q ) \partial _y (yE ) + (R - B ) E = 0 ,
\end{equation}   
for $E (y, T ; S , t ) $ or, component wise,   
$$
\partial _T e_{ij } - \frac{1 }{2 } \partial _y ^2 (y^2 \sigma _i (y, T )^2 e_{ij } ) + (r_i - q_i ) \partial _y (y e_{ij } ) + r_i e_{ij } - \sum _k b_{ik } e_{kj } = 0 , \ \ i, j = 1 , \dots , n .
$$
If $\Sigma ^2 $ is only $C^1 $, we can only integrate by parts once, but we still find that $E $ satisfies the forward equation in the weak sense:   
\begin{equation} \label{eq:FS_FW_weak}   
0 = \int _0 ^{\infty } \underline{F }(y) \left( \partial _T E  + (R - Q ) \partial _y (yE ) + (R - B ) E \right) dy + \frac{1 }{2 } \int _0 ^{\infty } (\partial _y \underline{F } ) \, \partial_y \left( y^2 \Sigma ^2 E \right) dy ,   
\end{equation}   
where, in view of the Gaussian estimates for $E (K, T ; S , t ) $ (see Appendix B) we can, by a density argument, take $\underline{F } $ to be any function in $H^1 _{\rm loc } (\mathbb{R } ) ^n $ which together with its derivative is of polynomial growth.   
\medskip   

\noindent We next use this to derive a Dupire-type system of PDEs for the column-vector of call prices      
\begin{equation} \nonumber   
\left( C_{1j } (K, T ; S , t ) , \ldots , C_{nj } (K, T ; S , t ) \right) = \left( C_j (K, \epsilon _1 , T ; S, e_j , t ) , \ldots C_j (K , \epsilon _n , T ; S, e_j , t ) \right) ,
\end{equation}
where we recall that $\epsilon _i (e_j ) = \delta _{ij } . $ By (\ref{eq:call_price_repr}) and (\ref{eq:FS_FW_weak}) with $\underline{F } (y) := \left( \max (y - K , 0 ) \delta _{ik } \right) _{1 \leq k \leq n } $,   
\begin{eqnarray*}   
&&\partial _T C_j (K, \epsilon _i , T ) = \int _{\mathbb{R }_{> 0 } } \, \partial_T e_{ij } (y, T ; S , t ) \max (y - K , 0 ) \, dy \\
&=& \int _{\mathbb{R }_{> 0 } } \left( - \frac{1 }{2 } \partial _y (y^2 \sigma _i ^2 e_{ij } ) + (r_i - q_i ) y e_{ij } \right)\, H(y - K ) dy - \int _{\mathbb{R }_{> 0 } }  \left(     
r_i e_{ij } - \sum _k b_{ik } e_{kj } \right) \max (y - K , 0 ) \, dy \\   
&=& \frac{1 }{2 } K^2 \sigma _i (K , T )^2 e_{ij } (K, T ) + (r_i - q_i ) \int _{\mathbb{R }_{> 0 } } y e_{ij } H(y - K ) dy - r_i C_j (K , \epsilon _i , T ) \\
&& + \sum _k b_{ik } C_j (K , \epsilon _k , T ) ,   
\end{eqnarray*}   
where $H(x) $ is the Heaviside function and where we used (\ref{eq:call_price_repr}) again for the last two terms. If we now for the first term we use the Breeden-Litzenberger relation (\ref{eq:BL}), and rewrite the second term as   
\begin{eqnarray*}
\int _{\mathbb{R }_{> 0 } } \, y \, e_{ij } (y, T ) \, H(y - K ) \, dy &=& \int _{\mathbb{R }_{> 0 } } \, e_{ij } (y, T ) \, \max (y - K , 0 ) \, dy + K \int _0 ^{\infty } \, e_{ij } (y, T ) \, H(y - K ) \, dy \\   
&=& C_j (K , \epsilon _i , T ) - K \partial _K C_j (K , \epsilon _i , T ) ,   
\end{eqnarray*}
we find the $(i, j ) $-th component of equation (\ref{eq:RS_Dupire1}), thereby proving theorem \ref{thm:RS_Dupire}.

\section{Proof of Lemma \ref{L4.1}}
We examine the non-vanishing of the matrix coefficients of the fundamental solution of a strictly parabolic Cauchy problem:
\begin{equation}\label{A.1}
\begin{array}{lll}
L(\tau,y;\p)\U(y,\tau):=\displaystyle \para{\p_\tau-\sum_{k=0}^2 \mathbb{A }_k(y,\tau)\p_y^k}\U(y,\tau)=0 & 0<\tau<T,\,y\in\R\cr
\U(y,0)=\U_0(y) & y\in\R ,
\end{array}
\end{equation}
where $\U(y,\tau)={}^t(u_1(y,\tau),\dots,u_d(y,\tau))$ and $\U_0(y)={}^t(u_{0,1}(y),\dots,u_{0,d}(y))$ are $d$-dimensional column vectors and where   
$$
\mathbb{A }_k(y,\tau)=(a^k_{ij}(y,\tau))_{1\leq i,j\leq d},\quad k=0,1,2
$$
are $d \times d $ matrix-valued functions, which are uniformly continuous in $(y,\tau)$ and H\"older continuous with exponent $\alpha > 0 $ in $y $, with $\mathbb{A } _2 (y , \tau ) $ strictly positive definite on $\mathbb{R } . $ We use blackboard bold letters for the coefficients to avoid confusion with the $A_k $ of sections 2 to 5. By Eidel'man \cite{Eidelman}, \cite{F} there exists a unique fundamental solution $E ( y, \tau ; y_0 , s )   =(e_{i,j } ( y, \tau ; y_0 , s ) )_{1\leq i,j\leq d } $ characterised by
\begin{equation}
\begin{array}{lll}
L(\tau,y,\p)E:=\displaystyle\para{ \p_\tau -\sum_{k=0}^2\mathbb{A }_k(y,\tau)\p_y^k}E( y, \tau ; y_0 , s ) =0, & 0\leq s <\tau<T,\, y, y_0 \in\R\cr
\displaystyle\lim_{\tau\to s^+}E( y, \tau ; y_0 , s )=\delta(y- y_0  )I_{d\times d}, & y, y_0 \in\R
\end{array}
\end{equation}
with $I_{d\times d} $ the $d\times d$ unit matrix, and whose matrix elements satisfy the Gaussian estimates   
$$
|\partial _y ^k e_{ij} (y, \tau ; y_0 , s ) | \leq \frac{C }{(\tau-s )^{(k + 1 )/2} } \exp\para{- c \frac{|y- y_0 |^2}{\tau-s} } , \ k = 0, 1, 2 ,
$$
and
$$
| \partial _{\tau } e_{ij} ( y, \tau ; y_0 , s ) | \leq \frac{C }{(\tau-s)^{3/2 } } \exp\para{- c \displaystyle\frac{|y - y_0 |^2}{\tau -s } } .
$$
for suitable constants $c , C > 0 . $ Otsuka \cite{Otsuka} determined necessary and sufficient conditions for the matrix elements of the fundamental solution to be non-negative:

\begin{theorem} (Otsuka \cite{Otsuka}) The fundamental solution of a strictly parabolic second order system (\ref{A.1}) has the positivity property
$$
e_{ij}( y, \tau ; y_0 , s ) \geq 0, \quad 0\leq s <\tau<T,\,x,y\in\R,\quad i,j=1,\dots,d.
$$
if and only if
\begin{enumerate}

\item [1. ] The first and second order coefficients $\mathbb{A }_k(y, \tau ) $ ($k = 1, 2 $) are diagonal.   

\item [2. ] For all $i\neq j$, $a^0_{ij}(y,\tau)\geq 0 . $

\end{enumerate}

\end{theorem}

We next examine strict positivity of the matrix coefficients of $E (y, \tau ; y_0, s_0 ) $, and will therefore assume conditions 1 and 2 to hold. We first decompose the operator $L$ into a diagonal part, $L_{\rm diag } $ and a 0-th order non-diagonal part, $\mathbb{B } $,   
$$
L = L_{\rm diag } + \mathbb{B } .   
$$
In particular, $\mathbb{A }_0 - \mathbb{B } $ is diagonal, so that if $\mathbb{B }(y, \tau ) = \left( (b_{ij}(y, \tau ) \right) _{i ,j } $ then $b_{ij}(y,\tau )=a^0_{ij}(y, \tau ) $ for $i \neq j $; one can, but does not need to, take $b_{ii}(y,\tau ) = 0 . $ Let
$$
E_{\rm diag } (y, \tau ;  y_0 , s  )=\textrm{Diag}(e_1 (y , \tau ;  y_0  , s ), \dots ,e_d (y, \tau ; y_0 , s ) )
$$
be the fundamental solution of $L_{\rm diag } . $ It is known that the $e_i $ satisfy a Gaussian lower bound: there exist positive constants $\delta_0$ and $\epsilon_0 $ such that 
\begin{equation}\label{e_j-positive}
e_i (y  , \tau ; y_0 , s ) \geq \delta_0 (\tau - s)^{-1/2}\exp\para{-\epsilon_0\frac{(y - y_0 ) ^2}{\tau-s}} :   
\end{equation}
see for example \cite{Aronson}. Recall that, by assumption, $b_{ij } (y, \tau ) \geq 0 $ for all $i $ and $j $, and let $B_* = (b_{*, ij } )_{i, j } $ be the matrix of the coefficient-wise infinimae of $B(y, \tau ) $:
$$
b_{*, ij } := \inf_{y  \in \mathbb{R }, \tau > 0 } b_{ij } (y, \tau ) .
$$
The following lower bound can be considered to be a sharpening of the sufficient part of Otsuka's theorem above. If we apply it to the system (\ref{eq:aux_PDE}) while taking $\mathbb{B } = B $, it immediately implies (\ref{5.3}), thereby completing the proof of this lemma.

\begin{theorem} \label{thm:LB_Phi} There exists constants $c, C > 0 $ such that for all $i, j $,
\begin{equation} \label{eq:LB_Phi}   
e_{ij } (y, \tau ; y_0 , s ) \geq C \left( e^{c (\tau - s ) B_* } \right) _{ij } \frac{e^{ - \epsilon _0 (y - z )^2 / (\tau - s ) } }{\sqrt{\tau - s } } ,
\end{equation}   
where we can in fact take $c = \delta _0 (\epsilon _0 ^{-1 } \pi )^{1/2 } $ and $C = \delta _0 . $
\end{theorem}
   
\begin{proof} Following \cite{Otsuka}, the fundamental solution $E(\tau ,s,  y , y_0  ) $ can be written as
$$
 E (y, \tau ; y_0 , s )    =E_{\rm diag } ( y, \tau ; y_0 , s ) + E_1 ( y , \tau ; y_0 , s ) ,
$$
where $E_{\rm diag } $ is the fundamental solution of the diagonal operator introduced above, and where the second part of the fundamental solution $E_1=(e_{1, ij}(y , \tau ; y_0 , s ) )_{1\leq i,j\leq d}$ is given by
\begin{equation} \label{eq:E_1}
E_1( y, \tau ; y_0 , s  ) =\int_s^\tau \! \!\int_\R E_{\rm diag } ( y, \tau ; z, t ) \Phi ( z, t ; y_0 , s ) \, dz dt ,
\end{equation}
where
\begin{equation} \label{eq:Phi}   
\Phi( y, \tau ; y_0 , s ) = \sum_{p\geq 1}\Phi_p( y, \tau ; y_0 , s) ,   
\end{equation}   
with   
\begin{equation} \label{eq:Phi_1}   
\Phi_1 ( y, \tau ; y_0 , s  ) =B(y, \tau ) E_{\rm diag } ( y, \tau ; y_0 , s ) ,   
\end{equation}
and
\begin{equation} \label{eq:Phi_p}
\Phi_p( y, \tau ; y_0 , s ) = \int_s^\tau\!\!\int_\R \Phi_1( y, \tau ; z, t ) \Phi_{p-1 } ( z, t ; y_0 , s ) dz dt .   
\end{equation}   
The series is absolutely convergent.   

\begin{lemma} \label{lemma:LB_Phi} For $p \geq 1 $,
\begin{equation} \label{eq:LB_Phi1}
\left( \Phi _p \right)_{i j }  (y , \tau ,  y_0 , s  ) \geq \delta _0 ^p \left( \frac{\pi }{\epsilon _0 } \right) ^{(p - 1 ) / 2 } \frac{(\tau - s )^{p - 1 } }{(p - 1 )! } \left( B_* ^p \right) _{i, j } \frac{\exp\para{-\epsilon_0\frac{\abs{y - y_0} ^2}{\tau-s}}}{\sqrt{\tau - s } } ,
\end{equation}
for $1\leq i, j \leq n . $
\end{lemma}

\begin{proof} By induction on $p . $ If $p = 1 $, the stated inequality is an immediate consequence of (\ref{eq:Phi_1}) and (\ref{e_j-positive}). Next, suppose that (\ref{eq:LB_Phi1}) holds for $p - 1 . $ Then, by (\ref{eq:Phi_p}) and (\ref{e_j-positive}),
\begin{eqnarray*}
&&\left( \Phi _p \right)_{ij } (y, \tau ; y_0 , s ) = \int _s ^{\tau } \int _{\mathbb{R } } \, \sum _k b_{i k } (y, \tau ) e_{0, k } (y , \tau ; z, t ) \left( \Phi _{p - 1 } \right)_{kj } (z, t ; y_0 , s ) \,dz dt \\   
&\geq & \delta _0 ^p \left(\epsilon _0 ^{-1 } \pi \right) ^{(p - 2 ) / 2 } \int _s ^{\tau } \, \frac{ (t - s )^{p - 2 } }{(p - 2 )! } \sum _k b_{*, ik } \left( B_* ^{p - 1 } \right)_{k j } \left( \int _{\mathbb{R } } \frac{e^{ - \epsilon _0 (y - z )^2 / (\tau - t ) } }   {\sqrt{\tau - t } }
\frac{e^{ - \epsilon _0 (z - y_0 )^2 / (t - s ) } }{\sqrt{t - s } } \, dz \right) dt \\
&=& \delta _0 ^p (\epsilon _0 ^{-1 } \pi )^{(p - 1 )/2 } \, (B_* ^p )_{ij } \, \frac{(\tau - s ) ^{p - 1 } }{(p - 1 )! } \frac{e^{ - \epsilon _0 (y - y_0 )^2 / (\tau - s ) } }{\sqrt{\tau - s } } .
\end{eqnarray*}
where we used the semi-group property of the classical heat-kernel\footnote{In the form: if $p_t (z) := t^{-1 / 2 } e^{- \epsilon _0 z^2 } $, then the convolution $p_{t_1 } * p_{t_2 } = \sqrt{\epsilon ^{-1 } \pi } p_{t_1 + t_2 } . $}.
\end{proof}
   
Returning to the proof of theorem \ref{thm:LB_Phi},  let us take $s = 0 $ to simplify notations. Then by (\ref{eq:E_1}), \ref{eq:Phi}) and (\ref{eq:LB_Phi1}), putting $c := \delta _0 (\epsilon _0 ^{-1 } \pi )^{1/2 } $,
\begin{eqnarray*}
e_{1 , ij } (y, \tau ; y_0 , 0 ) &\geq & \delta _0 \sum _{p \geq 1 } \int _0 ^{\tau } \int _{\mathbb{R } } e_{0, i } (y, \tau ; z, t ) c^{p - 1 } \frac{\tau ^{p - 1 } }{(p - 1 )! } \left( B_* ^p \right)_{ij } \frac{e^{ - \epsilon _0 (z - y_0 )^2 / t } }{\sqrt{t } } \, dz dt \\
&\geq & \delta _0 \left(\sum _{p \geq 1 } \frac{(c \tau )^p }{p! } B_* ^p \right) _{ij } \cdot \frac{e^{ - \epsilon _0 (y - y_0 )^2 / \tau } }{\sqrt{\tau } } ,
\end{eqnarray*}
where we used (\ref{e_j-positive}) and the semi-group property of the classical heat-kernel again. Adding $E_0 (y, \tau ; y_0 , 0 ) $, estimated from below by (\ref{e_j-positive}) times the identity matrix, the theorem follows.

\end{proof}

\section{Proof of lemma \ref{L6.2}}

Let us write the fundamental solution $E_{A_1 } (y, \tau ; z, s ) $ of the time-homogeneous operator $\partial _{\tau } - \mathcal{L }_{A_1 } $ as  $E (y, z, \tau - s ) = \left( e_{ij } (y, z , \tau - s ) \right)_{1 \leq i, j \leq n } . $    Then by Duhamel's principle,
\begin{equation} \label{eq:Duhamel}
\w (y, \tau ) = \int _0 ^{\tau } \int _{\mathbb{R } } E(y, z, s ) G(z) \V (z, \tau - s ) dz ds
\end{equation}
where we know that
\begin{equation}
|E (y, z, s ) | \leq C s^{-1/2 } e^{- c (y - z )^2 / s } \ \mbox{and } |E_y (y, z, s ) | \leq C s^{-1 } e^{- c (y - z )^2 / s } ,
\end{equation}
since $A_1 (y) $ is $C^1 $ and therefore certainly H\"older continuous. Here $E $ and $G = A_1 - A_2 $ are $n $ by $n $ matrices, and $\V = (v_1 , \ldots , v_n ) $ is a vector-valued function which is a column vector of the fundamental solution with pole in 0 of an auxilliary parabolic system: see (\ref{eq:aux_PDE}). Since $G $ is in fact diagonal, $G ={\rm diag } (G_1 , \ldots , G_n ) $, the $i $-th component of the integrand is $\sum _j e_{ij } (y, z, s ) v_j (z, \tau - s ) G_j (z) $, and if we now identify $G $ with the vector of its diagonal elements, we can interpret (\ref{eq:Duhamel}) as
\begin{equation}
\w (y, \tau ) = \int _0 ^{\tau } \, K_s (G ) (y) \, ds ,
\end{equation}
where $K_s $ is the integral operator on $\mathbb{R } $ with matrix-valued kernel $k_s (y, z ) = \left( e_{ij } (y, z, s ) v_j (z, \tau - s ) \right)_{i, j } $ acting on $L^2 (\mathbb{R } )^n $:   
\begin{equation}
K_s (G)(y ) := \int _{\mathbb{R} }  k_s (y, z) G(z) dz
\end{equation}
\medskip

\noindent 1. We begin by estimating the $L^2 $-norm of $\w (\cdot , \tau ) . $ Since $|| \w (\cdot , \tau ) ||_{L^2 (\mathbb{R } ) } \leq \int _0 ^{\tau } \, || K_s (G ) ||_{L^2 (\mathbb{R } ) } \, ds $, it suffices to estimate the $L^2 $-norm of each the operators $K_s . $ We will do using the following classical lemma of Schur, in a version for integral operators acting on vector-valued functions:   
   
\begin{lemma} \label{lemma:Schur} Let $K $ be an integral operator on $L^2 (\mathbb{R } ) $ with matrix-valued kernel $\left( k_{ij } (y, z ) \right) _{1 \leq i, j \leq n } . $ Then the operator-norm $|| K || $ of $K $ can be bounded by $|| K || \leq n \sqrt{C_1 C_2 } $, where
\begin{equation} \label{eq:Schur}
C_1 := C_1 (k) := \max _{i, j } \sup _y \int _{\mathbb{R } } \, | k_{ij } (y, z ) | \, dz , \ \ C_2 := C_2 (k) := \max _{i, j } \sup _z \int _{\mathbb{R } } \, | k_{ij } (y, z ) | \, dy .
\end{equation}   
\end{lemma}

By the Gaussian upper bounds for fundamental solutions (cf. Appendix B)  we know that
$$
|k_s (y, z ) | \leq C \frac{e^{ - c (y - z )^2 / s } }{\sqrt{s } } \frac{e^{- c z^2 / (\tau - s ) } }{\sqrt{\tau - s } } .
$$
Using the semi-group property of the fundamental solution of the classical heat equation, we then see that       
$$
C_1 (k_s ) = \sup _y C \frac{e^{- c y^2 / \tau } }{\sqrt{\tau } } = C/\sqrt{\tau } ,
$$
while an easy estimate shows that $C_2 (k_s ) \leq C / \sqrt{\tau - s } . $ Hence $|| K_s || \leq C (\tau (\tau - s ) )^{-1/4 } $, and therefore
\begin{equation} \label{eq:norm_w}
|| \w (\cdot , \tau ) ||_{L^2 (\mathbb{R } ) } \leq C \left( \int _0 ^{\tau } \left( \tau (\tau - s ) \right)^{-1 / 4 } ds \right) ||G ||_{L^2 (\mathbb{R } ) } = C \sqrt{\tau } \, ||G ||_{L^2 (\mathbb{R } ) } ,   
\end{equation}
with $C $ independent of $\tau . $
\medskip

\noindent 2. Next,
$$
\w _y (y , \tau ) = \int _0 ^{\tau } \, K^1 _s (G ) (y) \, ds ,
$$
where $K^1 _s $ is the integral operator with kernel
$$
k^1 _s (y, z ) = E_y (y, z, s ) v(z, \tau - s ) .
$$
We now have that
$$
|k^1 _s (y, z ) | \leq C \frac{e^{ - c (y - z )^2 / s } }{s } \frac{e^{- c z^2 / (\tau - s ) } }{\sqrt{\tau - s } } ,
$$
and therefore (simply multiply the previous estimates by $s^{-1 / 2 } $)
$$
C_1 (k^1 _s ) \leq \frac{C }{\sqrt{s \tau } } , \ \ C_2 (k^1 _s ) \leq \frac{C }{\sqrt{s (\tau - s ) } } .   
$$
It follows that   
\begin{equation} \label{eq:norm_w_y}
|| \w _y (\cdot , \tau ) ||_{L^2 (\mathbb{R } ) } \leq C \left( \int _0 ^{\tau } \frac{ds }{\left( \tau (\tau - s ) \right)^{1/4 } \sqrt{s } } \right) ||G ||_{L^2 (\mathbb{R } ) } = C ||G ||_{L^2 (\mathbb{R } ) } ,   
\end{equation}
with $C $ again independent of $\tau . $ Together with (\ref{eq:norm_w}), and with $\tau = \tau ^* $, this proves the first part of lemma \ref{L6.2}.   
\medskip

\noindent 3. We cannot in the same way show that $|| \w _{yy } (\cdot , \tau ) ||_{L^2 (\mathbb{R } ) } $ or $|| \w _{\tau } \cdot , \tau ) ||_{L^2 (\mathbb{R } ) } $ are bounded by $|| G ||_{L^2 (\mathbb{R } ) } $, since the singularity in $s = 0 $ becomes too strong: for example, $E_{yy } (y, z, s ) \leq C s^{-3/2 } e^{- c (y - z )^2 / s } $ we would end up with a $1/s $-singularity in (\ref{eq:norm_w_y}).  We therefore proceed differently by differentiating the PDE for $\w $, which gives   
$$
(\w_y )_{\tau } - \mathcal{L } (\w_y ) = G_y \V + G \V_y + Q_1 (\w ) ,
$$
with $Q_1 = [\partial _y , \mathcal{L } ] $ a first order differential operator. Write $\w_y = \w _1 + \w _2 + \w _3 $, where $\w _1 $, $\w _2 $ and $\w_3 $ solve the inhomogeneous PDE with right hand side $G_y \V $, $G \V _y $ and $\w _3 $ and $Q_1 (\w) $, respectively, and initial value 0 for $\tau = 0 . $ Then, by what we have already proven,   
$$
|| \w _1 (\cdot , \tau ) ||_{H^1 (\mathbb{R } ) } \leq C ||G_y ||_{L^2 (\mathbb{R } ) } .   
$$
We next show that
\begin{equation} \label{eq:norm_H^1_w_2}
||\w_2 (\cdot , \tau ) ||_{H^1 (\mathbb{R } ) } \leq C \tau ^{-1/2 } ||G ||_{L^2 (\mathbb{R } ) } .   
\end{equation}
Indeed, by Duhamel,
$$
\w_2 (y, \tau ) = \int _0 ^{\tau } \widetilde{K }_s (G) ds ,
$$
where $\widetilde{K }_s $ is the integral operator with kernel $\widetilde{k }_s (y, z ) = \left( e_{ij } ((y, z, s ) v_{j , z } (z, \tau - s ) \right)_{i, j } . $ Hence, the norm of the kernel and of its derivative with resepct to $y $ can be bounded by
$$
| \widetilde{k }_s (y, z ) | \leq C \frac{e^{ - c (y - z )^2 / s } }{\sqrt{s } } \frac{e^{- c z^2 / (\tau - s ) } }{\tau - s } ,
$$
and
$$
| \partial _y \widetilde{k }_s (y, z ) | \leq C \frac{e^{ - c (y - z )^2 / s } }{s } \frac{e^{- c z^2 / (\tau - s ) } }{\tau - s } .
$$
Proceeding as before, we find that
$$
|| \w _2 (\cdot , \tau ) ||_{L^2 (\mathbb{R } ) } \leq C \int _0 ^{\tau } \frac{ds }{\tau ^{1/4 } (\tau - s )^{3/4 } } ||G ||_{L^2 (\mathbb{R } ) } = C \, ||G ||_{L^2 (\mathbb{R } ) }   
$$
and
$$
|| \w _{2, y } ||_{L^2 (\mathbb{R } ) } \leq C \int _0 ^{\tau } \frac{ds }{\tau ^{1/4 } s^{1/2 } (\tau - s )^{3/4 } } ||G ||_{L^2 (\mathbb{R } ) } = (C / \sqrt{\tau } ) ||G ||_{L^2 (\mathbb{R } ) } ,   
$$
with constants $C $ which are independent of $\tau $, which proves (\ref{eq:norm_H^1_w_2}).

\bigskip

Finally, to bound the norm of $\w _3 $ we use the classical energy estimate:
\begin{eqnarray*}
|| \w _3 (\cdot , \tau ) ||_{H^1 (\mathbb{R } ) } ^2 &\leq & C || Q_1 (w) ||_{L^2 ((0, T ) \times \mathbb{R } ) } ^2 \\
&\leq & C \left( \int_0 ^{\tau } \, || \w (\cdot , s ) ||_{H^1 } \, ds \right) \\
&\leq & C \tau ||G ||_{L^2 (\mathbb{R } ) } ^2 ,
\end{eqnarray*}
by subsection 2.3 again. In conclusion, we have shown that   
$$
|| \w _y ||_{H^1 (\mathbb{R } ) } \leq C \max (\tau ^{-1/2 } , \tau ) \, \left( ||G ||_{L^2 (\mathbb{R } ) } + ||G_y ||_{L^2 (\mathbb{R } ) } \right) ,   
$$
which completes the proof of lemma \ref{L6.2} when specializing to $\tau = \tau ^*$.   
   

\end{document}